\documentclass{amsart}
\usepackage{amsfonts}
\usepackage{bbm}
\usepackage{mathrsfs}
\usepackage{amsmath}
\usepackage{bigdelim}
\usepackage{multirow}
\usepackage{amssymb}
\usepackage{indentfirst}
\usepackage{fancyhdr}
\usepackage{amscd,xy,xypic,amssymb,amsmath,graphicx,caption,verbatim}
\usepackage[TS1,OT1,T1]{fontenc}
\newtheorem{theorem}{Theorem}[section]
\newtheorem{lemma}[theorem]{Lemma}
\newtheorem{corollary}[theorem]{Corollary}
\newtheorem{proposition}[theorem]{Proposition}

\theoremstyle{definition}
\newtheorem{definition}[theorem]{Definition}

\theoremstyle{remark}
\newtheorem{remark}[theorem]{Remark}
\numberwithin{equation}{section}

\newcommand\hH{\widehat{H}}
\newcommand\hb{\widehat{b}}
\newcommand\hT{\widehat{T}}

\newcommand\bH{\overline{H}}
\newcommand\bb{\overline{b}}

\newcommand\tM{\widetilde{M}}

\begin{document}

\title[Relative $L$-theory and signature]{On the relative $L$-theory and the relative signature of PL manifolds with boundary}
	\author[B. Hou]{Bingzhe Hou}
	\address{School of Mathematics, Jilin University, 130012, Changchun, P. R. China}
    \email{houbz@jlu.edu.cn}
	\author[H. Liu]{Hongzhi Liu}
	\address{School of Mathematics, Shanghai University of Finance and Economics, 200433, Shanghai, P. R. China.}
    \email{liu.hongzhi@sufe.edu.cn}
	\date{}
\begin{abstract}
In this paper, we give a new description of the group structure of the relative structure group of PL manifolds with boundary, and obtain a surgery exact sequence in the category of groups. Then we focus on the relative $L$-group of PL manifolds with boundary, and  map it to the $K$-theory additively.
% \texttt{testmath.tex} file from the \textsf{amsmath} distribution.  
\end{abstract}

\maketitle

\section{Introduction}\label{sec of rel introduction}

In this paper, we give a new description of the group structure of the relative structure group of PL manifolds with boundary, and obtain a surgery exact sequence in the category of groups. Then we focus on the relative $L$-group of PL manifolds with boundary, and  map it to the $K$-theory additively.
	
The surgery exact sequence and the relative surgery exact sequence are powerful tools to study the classification of PL manifolds and PL manifolds with boundary (Wall \cite{Wall70}, Quinn \cite{Quinn1971andthesutgeryobsturction}, Ranicki \cite{Ranicki92book}). Originally, they were defined as  exact sequences of groups and sets.  In  \cite{WXY18}, Weinberger, Xie and Yu showed that the surgery exact sequence of PL manifolds is actually an exact sequence consists of groups and homomorphisms by introducing a new definition of the structure group of PL manifolds based on ideas of Wall and ideas from the
controlled topology, which leads to a transparent group structure of the topological structure group given by disjoint union.   Our first main result, is to generalize Weinberger, Xie and Yu's result to the relative surgery exact sequence.  We give a new definition of the relative structure group of PL manifolds with boundary, whose group structure is as transparent as the disjoint union, and put the relative $L$-group of  PL manifolds with boundary into an exact sequence of groups.
More precisely, let $(X,\partial X)$ be an $n$-dimensional PL manifold with boundary, set $\Gamma=\pi_1 X, G=\pi_1(\partial X)$. Then the relative $L$-group of $(X, \partial X)$ is denoted as $L_n(\pi_1X, \pi_1(\partial X)),$ the relative normal group is denoted as $N_{n}(X, \partial X; \omega),$ and the relative structure group we define in the paper is denoted as  $S_{n}(X, \partial X; \omega).$   Then we have

\vspace{0.3cm}
{\bf Main Theorem 1.}(Theorem \ref{Les})
		We have the following long exact sequence of commutative groups
	\begin{align*}
		\cdots &\stackrel{}{\longrightarrow} S_{n}(X, \partial X; \omega)  \stackrel{\widetilde{\partial}_*}{\longrightarrow} N_{n}(X, \partial X; \omega) \stackrel{i_*}{\longrightarrow} L_{n}(\pi_1X, \pi_1(\partial X); \omega) \\
		&\stackrel{j_*}{\longrightarrow} S_{n-1}(X, \partial X; \omega)  \stackrel{\widetilde{\partial}_*}{\longrightarrow} N_{n-1}(X, \partial X; \omega) \stackrel{}{\longrightarrow} \cdots.
		\end{align*}
\vspace{0.2cm}

It is well known that there is a group homomorphism from the $L$-group to the $K$-theory of the Roe algebra, a geometric $C^*$-algebra. Then it is natural to ask whether we can define an additive map from the relative $L$-group to the $K$-theory of a certain geometric $C^*$-algebra. Let $(X, \partial X)$ be as above, and $\Gamma=\pi_1 X,\ G=\pi_1(\partial X).$ Let $\widetilde X$ (resp. $\widetilde{\partial X}$) be the universal covering of $X$ (resp. $\partial X$.) In \cite{CWY15},  Chang, Weinberger and Yu defined the relative Roe algebra, denoted as $C^*(\widetilde X, \widetilde{\partial X})^{\Gamma, G},$ and the relative index of the Dirac type operator on a manifold with boundary, which lives in the $K$-theory of   $C^*(\widetilde X, \widetilde{\partial X})^{\Gamma, G}.$ The relative index defined by Chang, Weinberger and Yu, can be viewed as the explanation of the bordism invariance of the index of the Dirac type operator. In this paper, inspired by Higson and Roe's constructions in  \cite{HR051,HR052,HR053}, we define the relative index of the signature operator on manifolds with boundary by the simplicial approach, which is denoted as $\text{relInd}(X, \partial X)$ for the PL manifold with boundary $(X, \partial X).$ This allows us to  consider the PL manifolds with boundary, apparently on which there is no signature operator, and define the additive map from the relative $L$-group to the $K$-theory of the relative Roe algebra.

\vspace{0.3cm}
{\bf Main Theorem 2.}(Theorem \ref{add map from L to K})
The map
\[
\text{relInd}: L_n(\pi_1(X),\pi_1(\partial X))\to K_n(C^*(\widetilde X, \widetilde{\partial X})^{\Gamma,G})
\]
is a well defined group homomorphism.
\vspace{0.2cm}

We mention that the relative index of the signature operator on a manifold with boundary has been used to prove the relative Novikov conjecture (\cite{DG171}, \cite{DG172}, \cite{Ttheis19}), but we are not aware of whether the relative index of  signature operator considered in those papers   are equal to the one we define here.

	This paper is organized as follows. In Section \ref{sec Surgery}, we generalize Weinberger, Xie, and Yu's results in \cite{WXY18}, to give a new description of the relative topological structure group of a topological manifold with boundary, and put the relative $L$-group into an exact sequence consists of groups. In Section \ref{sec K theory preparation}, we recall the definitions of the relative Roe algebra. In Section \ref{sec relative index}, we define the  relative signature of PL manifolds with boundary, and show that it induces an additive map from the relative $L$-group to the $K$-theory of the relative Roe algebra.
	
	The authors would like to thank Shmuel Weinberger, Zhizhang Xie and Guoliang Yu for their helpful guidance and advice. The second author is partially supported by NSFC 11901374.
	
\section{Surgery}\label{sec Surgery}

In this section, we give a new description of the relative surgery group and the relative surgery exact sequence, which could be viewed as a generalization of Weinberger, Xie and Yu's definition of structure groups of PL manifolds to the relative case.

We first recall some definitions related to the infinitesimally controlled homotopy equivalence.

Let $X$ be a closed topological manifold. Fix a metric on $X$ that agrees with the topology of $X$.

\begin{definition}
Let $Y$ be a topological space. We call a continuous map $\phi: Y\to X$ a control map of $Y$.
\end{definition}

\begin{definition}
Let $Y$ and $Z$ be two compact Hausdorff spaces equipped with control maps $\psi: Y\to X$ and $\phi: Z\to X$. A continuous map $f: Y\to Z$ is said to be a controlled homotopy equivalence over $X$, if
 \begin{enumerate}
 \item $\phi=\psi f$;
 \item there exists a continuous map $g: Z\to Y$ such that $\psi=\phi g$;
 \item $fg \sim_h {\rm I}_{Y}$ and $gf\sim_h {\rm I}_Z$.
\end{enumerate}
 \end{definition}

Now let us recall the definition of infinitesimally controlled homotopy equivalence (cf. \cite[Definition 3.3]{WXY18}).
\begin{definition}[Infinitesimally controlled homotopy equivalence]
Let $Y$ and $Z$ be two compact Hausdorff spaces equipped with control maps $\psi: Y\to X$ and $\phi: Z\to X$. A continuous map $f: Y\to Z$ is said to be an infinitesimally controlled homotopy equivalence over $X$, if there exist proper continuous maps
\begin{eqnarray*}
\Phi: Z\times [1, \infty)\to X\times [1, \infty) & \text{ and } & \Psi:  Y\times [1, \infty)\to X\times [1, \infty),\\
F: Y\times [1, \infty)\to Z\times [1, \infty)& \text{ and } & Z\times [1, \infty)\to Y\times [1, \infty)
\end{eqnarray*}
satisfying the following conditions:
\begin{enumerate}
\item $\Phi F=\Psi$;
\item $F|_{Y\times \{1\}}=f, \Phi|_{Z\times \{1\}}=\phi, \Psi|_{Y\times \{1\}}=\psi $;
\item there is a proper continuous homotopy $\{H_s\}_{0\leq s\leq 1}$ between
\[
H_0= FG \text{ and } H_1= \text{id}: Z\times [1, \infty)\to Z\times [1, \infty )
\]
such that the diameter of the set $\Phi(H(z,t))= \{\Phi(H_s(z, t)) | 0\leq s\leq 1 \}$
goes uniformly (i.e. independent of $z\in Z$ ) to zero, as $t\to \infty$;
\item there is a proper continuous homotopy $\{R_s\}_{0\leq s\leq 1}$ between
\[
R_0= GF \text{ and } H_1= \text{id}: Y\times [1, \infty)\to Y\times [1, \infty )
\]
such that the diameter of the set $\Psi(R(y,t))= \{\Psi(R_s(y, t)) | 0\leq s\leq 1 \}$
goes uniformly (i.e. independent of $y\in Y$ ) to zero, as $t\to \infty$;
\end{enumerate}
\end{definition}

    Let $X$ be a compact manifold with boundary $\partial X$ whose dimension is greater than $5$.
	The  definition of relative $L$-group follows from Wall's work in \cite{Wall70}.
	\begin{definition}[Objects for the definition of $L_n(\pi_1X, \pi_1(\partial X); \omega)$]\label{def ele of rel L}
		An object
		$$
		\theta=\{M, \partial_{\pm} M, \phi, N, \partial_{\pm} N, \psi, f\}
		$$
		in $L_n(\pi_1X, \pi_1(\partial X); \omega)$ consists of the following data

		\begin{figure}[htb]
			\centering
			\includegraphics[width=5in]{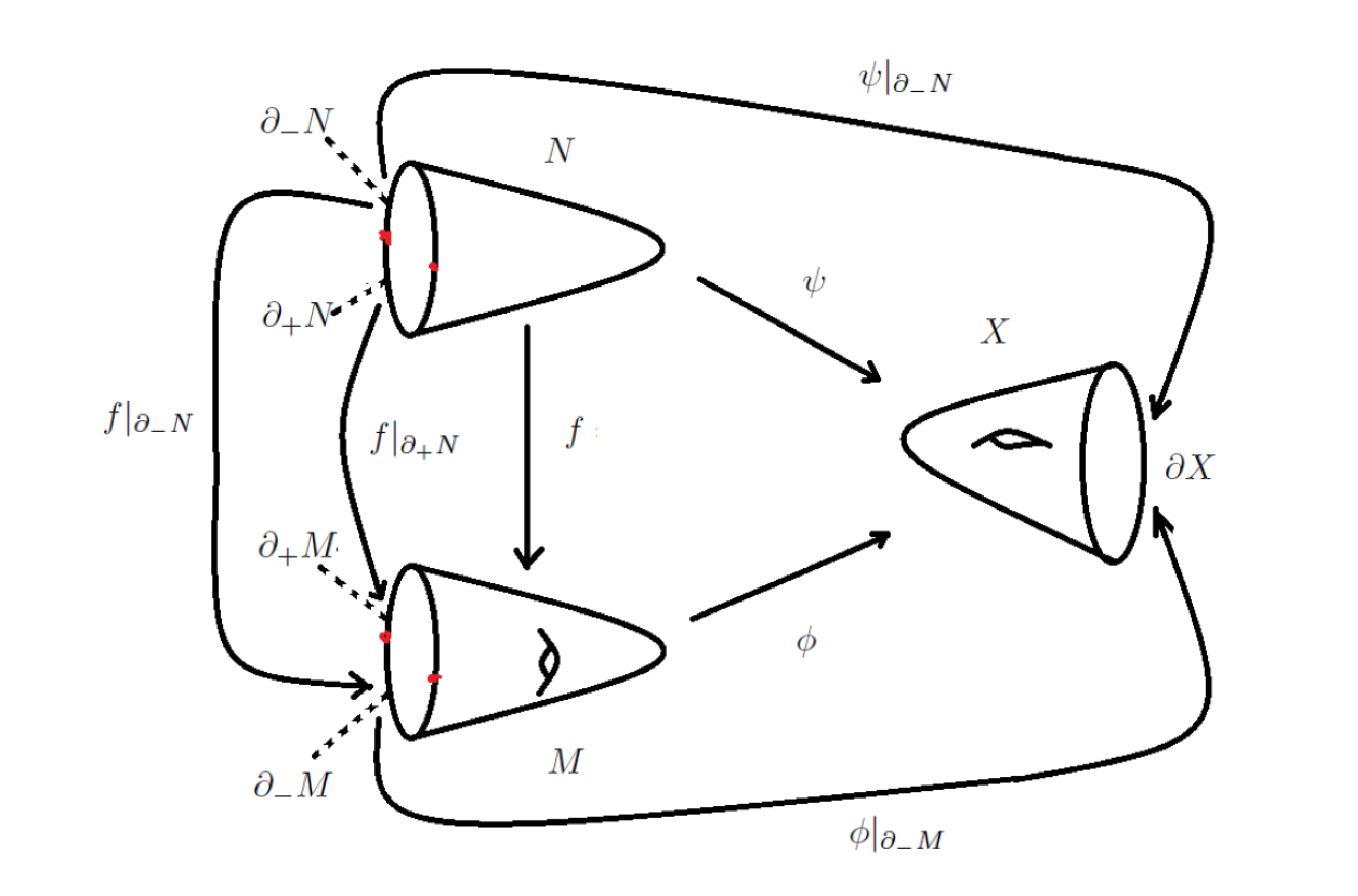}
			\caption{\footnotesize An object $\theta=\{M, \partial_{\pm} M, \phi, N, \partial_{\pm} N, \psi, f\}$ in $L_n(\pi_1X, \pi_1(\partial X); \omega)$.}
		\end{figure}

		\begin{enumerate}
			\item two manifold 2-ads $(M, \partial_{\pm}M)$ and $(N, \partial_{\pm}N)$ with ${\rm dim}M={\rm dim}N=n$, with $\partial M=\partial_+ M\cup \partial_- M$ (resp. $\partial N=\partial_+ N\cup \partial_- N$) the boundary of $M$ (resp. $\partial N$). In particular, $\partial_+ M \cap \partial_-M= \partial\partial_{\pm} M $ and $\partial_+ N \cap \partial_-N= \partial\partial_{\pm} N$;
			\item continuous maps $\phi: (M, \partial_{-} M)\rightarrow (X, \partial X)$ and $\psi: (N, \partial_{-} N)\rightarrow (X, \partial X)$ so that $\phi^*(\omega)$ and $\psi^*(\omega)$ describe the orientation characters of  $M$ and $N$;
			\item a degree one normal map of manifold 2-ads $f:(N, \partial_{\pm} N) \rightarrow (M, \partial_{\pm} M)$ such that  $\phi\circ f=\psi$;
			\item the restriction $f|_{\partial_+ N}: (\partial_+ N, \partial \partial_+ N) \to (\partial_+ M, \partial \partial_+ M)$ is a homotopy equivalence of pairs over $(X, \partial X)$;
			\item the restriction $f|_{\partial_- N}: \partial_- N\to \partial_- M$ is a degree one normal map over $\partial X$.
		\end{enumerate}
	\end{definition}

	\begin{definition}[Equivalence relation for the definition of $L_n(\pi_1X, \pi_1(\partial X); \omega)$]\label{equi rel L}
		Let
		$$
		\theta=\{M, \partial_{\pm} M, \phi, N, \partial_{\pm} N, \psi, f\}
		$$
		be an object in $L_n(\pi_1X, \pi_1(\partial X); \omega)$. We write $\theta\sim 0$ if the following conditions are satisfied.
		
 \begin{figure}[htb]
			\centering
			\includegraphics[width=5.0in]{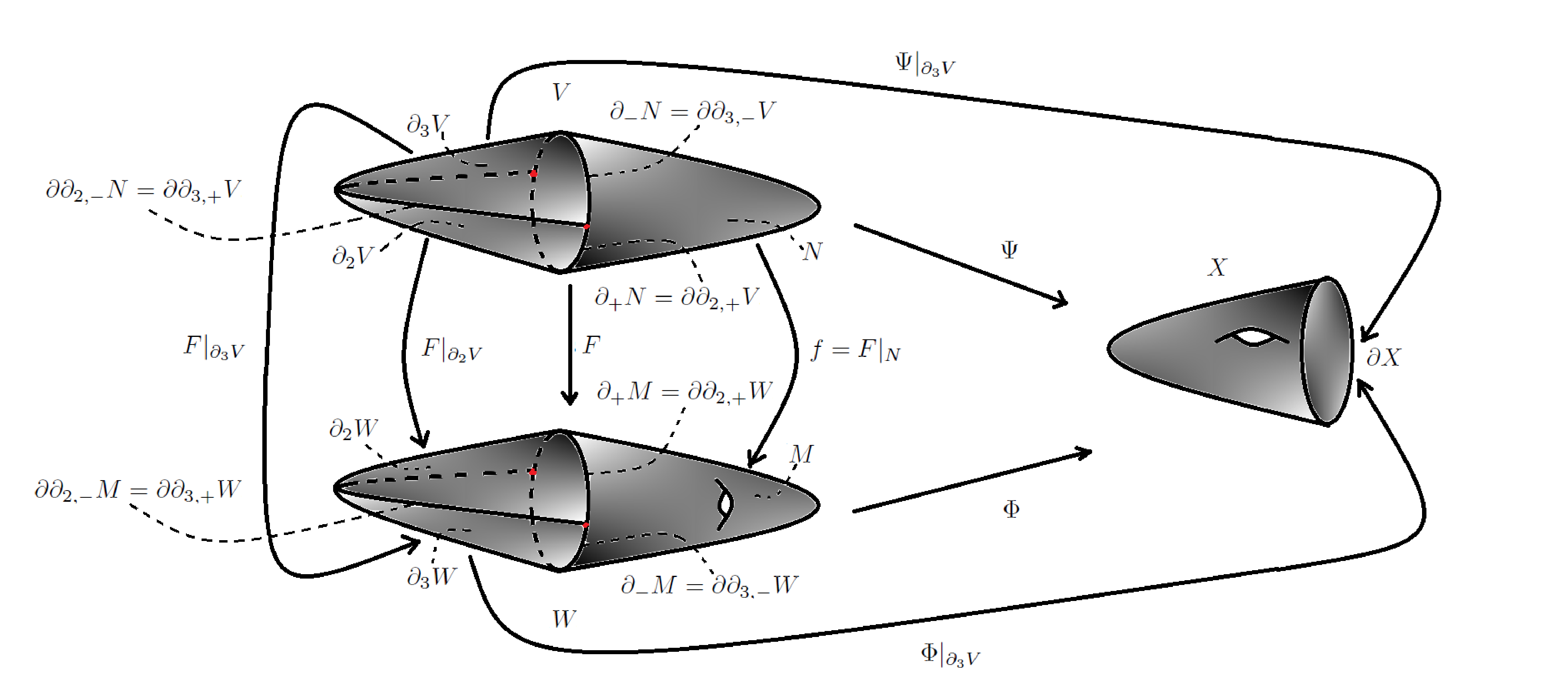}
			\caption{\footnotesize Equivalence relation $\theta\sim 0$ for the definition of $L_n(\pi_1X, \pi_1(\partial X); \omega)$.}
		\end{figure}
		
\begin{enumerate}
			\item There exists a manifold 3-ads $(W, \partial W)$ of dimension $(n+1)$ with a continuous map $\Phi: (W, \partial_3 W)\rightarrow (X, \partial X)$ so that $\Phi^*(\omega)$ describes the orientation character of $W$, where $\partial W=M(=\partial_1 W)\cup\partial_2 W\cup \partial_3 W$. Moreover, we have decompositions $\partial M=\partial_+ M\cup \partial_- M$, $\partial(\partial_2 W)=\partial\partial_{2, +} W\cup \partial\partial_{2, -} W$, and $\partial(\partial_3 W)=\partial\partial_{3, +} W\cup \partial\partial_{3, -} W$ such that
			$$
			\partial_+M=\partial\partial_{2, +} W, \ \partial_-M=\partial\partial_{3, -} W \ \text{and} \ \partial\partial_{2,-}M=\partial\partial_{3, +} W.
			$$
			Furthermore, we have
			$$
			\partial_+M\cap \partial_-M=\partial\partial_{2, +} W\cap \partial\partial_{2, -} W=\partial\partial_{3, +} W\cap \partial\partial_{3, -} W.
			$$
			\item Similarly, we have a manifold 3-ads $(V, \partial V)$ of dimension $(n+1)$ with a continuous map $\Psi: (V, \partial_3 V)\rightarrow (X, \partial X)$ so that $\Psi^*(\omega)$ describes the orientation character of $V$, where $\partial V=N(=\partial_1 V)\cup\partial_2 V\cup \partial_3 V$ satisfying similar conditions as $W$.
			
			\item There is a degree one normal map of manifold 3-ads $F: (V, \partial V) \rightarrow (W, \partial W)$ such that  $\Phi\circ F=\Psi$. Moreover, $F$ restricts to $f$ on $N\subseteq \partial V$.
			\item The restriction $F|_{\partial_2 V}: \partial_2 V\to \partial_2 W$ is a homotopy equivalence over $X$.
		\end{enumerate}
	\end{definition}

	We denote by $L_n(\pi_1X, \pi_1(\partial X); \omega)$ the set of equivalence classes from Definition \ref{equi rel L}. Note that $L_n(\pi_1X, \pi_1(\partial X); \omega)$  is an abelian group with the sum operation being disjoint union. We call $L_n(\pi_1X, \pi_1(\partial X); \omega)$  the relative $L$-group.

In the following, we give a controlled version of $L_n(\pi_1X, \pi_1(\partial X); \omega)$.
	
	\begin{definition}[Objects for the definition of $N_n(X, \partial X; \omega)$]\label{def ele of rel N}
		An object
		$$
		\theta=\{M, \partial_{\pm} M, \phi, N, \partial_{\pm} N, \psi, f\}
		$$
		in $N_n(X, \partial X; \omega)$ consists of the following data
		\begin{enumerate}
			\item two manifold 2-ads $(M, \partial_{\pm}M)$ and $(N, \partial_{\pm}N)$ with ${\rm dim}M={\rm dim}N=n$, with $\partial M=\partial_+ M\cup \partial_- M$ (resp. $\partial N=\partial_+ N\cup \partial_- N$) the boundary of $M$ (resp. $\partial N$). In particular, $\partial_+ M \cap \partial_-M= \partial\partial_{\pm} M $ and $\partial_+ N \cap \partial_-N= \partial\partial_{\pm} N$;
			\item continuous maps $\phi: (M, \partial_{-} M)\rightarrow (X, \partial X)$ and $\psi: (N, \partial_{-} N)\rightarrow (X, \partial X)$ so that $\phi^*(\omega)$ and $\psi^*(\omega)$ describe the orientation characters of  $M$ and $N$;
			\item a degree one normal map of manifold 2-ads $f:(N, \partial_{\pm} N) \rightarrow (M, \partial_{\pm} M)$ such that  $\phi\circ f=\psi$;
			\item the restriction $f|_{\partial_+ N}: \partial_+ N\to \partial_+ M$ is an infinitesimally controlled homotopy equivalence over $X$;
			\item the restriction $f|_{\partial_- N}: \partial_- N\to \partial_- M$ is a degree one normal map over $X$.
		\end{enumerate}
		
	\end{definition}
	
	\begin{definition}[Equivalence relation for the definition of $N_n(X, \partial X; \omega)$]\label{equi rel N}
		Let
		$$
		\theta=\{M, \partial_{\pm} M, \phi, N, \partial_{\pm} N, \psi, f\}
		$$
		be an object in $N_n(X, \partial X; \omega)$. We write $\theta\sim 0$ if the following conditions are satisfied.
		\begin{enumerate}
			\item There exists a manifold 3-ads $(W, \partial W)$ of dimension $(n+1)$ with a continuous map $\Phi: (W, \partial_3 W)\rightarrow (X, \partial X)$ so that $\Phi^*(\omega)$ describes the orientation character of $W$, where $\partial W=M(=\partial_1 W)\cup\partial_2 W\cup \partial_3 W$. Moreover, we have decompositions $\partial M=\partial_+ M\cup \partial_- M$, $\partial(\partial_2 W)=\partial\partial_{2, +} W\cup \partial\partial_{2, -} W$, and $\partial(\partial_3 W)=\partial\partial_{3, +} W\cup \partial\partial_{3, -} W$ such that
			$$
			\partial_+M=\partial\partial_{2, +} W, \ \partial_-M=\partial\partial_{3, -} W \ and \ \partial\partial_{2,-}M=\partial\partial_{3, +} W.
			$$
			Furthermore, we have
			$$
			\partial_+M\cap \partial_-M=\partial\partial_{2, +} W\cap \partial\partial_{2, -} W=\partial\partial_{3, +} W\cap \partial\partial_{3, -} W.
			$$
			\item Similarly, we have a manifold 3-ads $(V, \partial V)$ of dimension $(n+1)$ with a continuous map $\Psi: (V, \partial_3 V)\rightarrow (X, \partial X)$ so that $\Psi^*(\omega)$ describes the orientation character of $V$, where $\partial V=N(=\partial_1 V)\cup\partial_2 W\cup \partial_3 W$ satisfying similar conditions as $W$.
			
			\item There is a degree one normal map of manifold 3-ads $F: (V, \partial V) \rightarrow (W, \partial W)$ such that  $\Phi\circ F=\Psi$. Moreover, $F$ restricts to $f$ on $N\subseteq \partial V$.
			\item The restriction $F|_{\partial_2 V}: \partial_2 V\to \partial_2 W$ is an infinitesimally controlled homotopy equivalence over $X$.
		\end{enumerate}
		
	\end{definition}
	
	We denote by $N_n(X, \partial X; \omega)$ the set of equivalence classes from Definition \ref{equi rel N}, which is actually an abelian group with the sum operation being disjoint union. % We call this group as the relative normal group.
	
	Now we introduce the new description of relative topological surgery group.

	\begin{definition}[Objects for the definition of $S_n(X, \partial X; \omega)$]\label{def ele of rel S}
		An object
		$$
		\theta=\{M, \partial_{\pm} M, \phi, N, \partial_{\pm} N, \psi, f)\}
		$$
		in $S_n(X, \partial X; \omega)$ consists of the following data
		\begin{enumerate}
			\item two manifold 2-ads $(M, \partial_{\pm}M)$ and $(N, \partial_{\pm}N)$ with ${\rm dim}M={\rm dim}N=n$, with $\partial M=\partial_+ M\cup \partial_- M$ (resp. $\partial N=\partial_+ N\cup \partial_- N$) the boundary of $M$ (resp. $\partial N$). In particular, $\partial_+ M \cap \partial_-M= \partial\partial_{\pm} M $ and $\partial_+ N \cap \partial_-N= \partial\partial_{\pm} N$;
			\item continuous maps $\phi: (M, \partial_{-} M)\rightarrow (X, \partial X)$ and $\psi: (N, \partial_{-} N)\rightarrow (X, \partial X)$ so that $\phi^*(\omega)$ and $\psi^*(\omega)$ describe the orientation characters of  $M$ and $N$;
			\item a homotopy equivalence of manifold 2-ads $f:(N, \partial_{\pm} N) \rightarrow (M, \partial_{\pm} M)$ such that  $\phi\circ f=\psi$;
			\item the restriction $f|_{\partial_+ N}: \partial_+ N\to \partial_+ M$ is an infinitesimally controlled homotopy equivalence over $X$;
			\item the restriction $f|_{\partial_- N}: \partial_- N\to \partial_- M$ is a homotopy equivalence over $X$.
		\end{enumerate}	
	\end{definition}
	
	\begin{definition}[Equivalence relation for the definition of $S_n(X, \partial X; \omega)$]\label{equi rel S}
		Let
		$$
		\theta=\{M, \partial_{\pm} M, \phi, N, \partial_{\pm} N, \psi, f\}
		$$
		be an object in $S_n(X, \partial X; \omega)$. We write $\theta\sim 0$ if the following conditions are satisfied.
		\begin{enumerate}
			\item There exists a manifold 3-ads $(W, \partial W)$ of dimension $(n+1)$ with a continuous map $\Phi: (W, \partial_3 W)\rightarrow (X, \partial X)$ so that $\Phi^*(\omega)$ describes the orientation character of $W$, where $\partial W=M(=\partial_1 W)\cup\partial_2 W\cup \partial_3 W$. Moreover, we have decompositions $\partial M=\partial_+ M\cup \partial_- M$, $\partial(\partial_2 W)=\partial\partial_{2, +} W\cup \partial\partial_{2, -} W$, and $\partial(\partial_3 W)=\partial\partial_{3, +} W\cup \partial\partial_{3, -} W$ such that
			$$
			\partial_+M=\partial\partial_{2, +} W, \ \partial_-M=\partial\partial_{3, -} W \ \text{and} \ \partial\partial_{2,-}M=\partial\partial_{3, +} W.
			$$
			Furthermore, we have
			$$
			\partial_+M\cap \partial_-M=\partial\partial_{2, +} W\cap \partial\partial_{2, -} W=\partial\partial_{3, +} W\cap \partial\partial_{3, -} W.
			$$
			\item Similarly, we have a manifold 3-ads $(V, \partial V)$ of dimension $(n+1)$ with a continuous map $\Psi: (V, \partial_3 V)\rightarrow (X, \partial X)$ so that $\Psi^*(\omega)$ describes the orientation character of $V$, where $\partial V=N(=\partial_1 V)\cup\partial_2 V\cup \partial_3 V$ satisfying similar conditions as $W$.
			
			\item There is a homotopy equivalence of manifold 3-ads $F: (V, \partial V) \rightarrow (W, \partial W)$ such that  $\Phi\circ F=\Psi$. Moreover, $F$ restricts to $f$ on $N\subseteq \partial V$.
			\item The restriction $F|_{\partial_2 V}: \partial_2 V\to \partial_2 W$ is an infinitesimally controlled homotopy equivalence over $X$.
		\end{enumerate}
	\end{definition}
	
	We denote by $S_n(X, \partial X; \omega)$ the set of equivalence classes from Definition \ref{equi rel S}. It is not difficult to see that $S_n(X, \partial X; \omega)$  is an abelian group with the sum operation being disjoint union. %The abelian group $S_n(X, \partial X; \omega)$

	We need the following auxiliary group to form the new description of the relative surgery exact sequence.

	\begin{definition}[Objects for the definition of $L_n(\pi_1X, \pi_1(\partial X), X; \omega)$]\label{def ele of rel LX}
		An object
		$$
		\theta=\{M, \partial_{k} M, \phi, N, \partial_{k} N, \psi, f; k=1,2,3.\}
		$$
		in $L_n(\pi_1X, \pi_1(\partial X), X; \omega)$ consists of the following data
		\begin{enumerate}
			\item two manifold 3-ads $(M, \partial_{k}M; k=1,2,3.)$ and $(N, \partial_{k}N; k=1,2,3.)$ with ${\rm dim}M={\rm dim}N=n$, with $\partial M=\partial_1 M\cup \partial_2 M\cup \partial_3 M$ (resp. $\partial N=\partial_1 N\cup \partial_2 N\cup \partial_3 N$) the boundary of $M$ (resp. $\partial N$). Moreover, $\partial(\partial_i M)=\cup_{j\neq i}\partial\partial_{i,j}M$ for each $i=1,2,3$ and $\partial\partial_{i,j}M=\partial_i M\cap \partial_j M$ for any $i\neq j$;
			\item continuous maps $\phi: (M, \partial_{3} M)\rightarrow (X, \partial X)$ and $\psi: (N, \partial_{3} N)\rightarrow (X, \partial X)$ so that $\phi^*(\omega)$ and $\psi^*(\omega)$ describe the orientation characters of  $M$ and $N$;
			\item a degree one normal map of manifold 3-ads $f:(N, \partial N) \rightarrow (M, \partial M)$ such that  $\phi\circ f=\psi$;
			\item the restriction $f|_{\partial_1 N}: \partial_1 N\to \partial_1 M$ is a degree one normal map over $X$ ;
			\item the restriction $f|_{\partial_2 N}: \partial_2 N\to \partial_2 M$ is a homotopy equivalence over $X$ and it restricts to an infinitesimally controlled homotopy equivalence $f|_{\partial\partial_{1,2} N}: \partial\partial_{1,2} N\to \partial\partial_{1,2} M$ over $X$;
			\item the restriction $f|_{\partial_3 N}: \partial_3 N\to \partial_3 M$ is a degree one normal map over $X$.
		\end{enumerate}
		
	\end{definition}
	
	\begin{definition}[Equivalence relation for the definition of $L_n(\pi_1X, \pi_1(\partial X), X; \omega)$]\label{equi rel LX}
		Let
		$$
		\theta=\{M, \partial_{k} M, \phi, N, \partial_{k} N, \psi, f; k=1,2,3.\}
		$$
		be an object in $L_n(\pi_1X, \pi_1(\partial X), X; \omega)$. We write $\theta\sim 0$ if the following conditions are satisfied.
		\begin{enumerate}
			\item There exists a manifold 4-ads $(W, \partial W)$ of dimension $(n+1)$ with a continuous map $\Phi: (W, \partial_4 W)\rightarrow (X, \partial X)$ so that $\Phi^*(\omega)$ describes the orientation character of $W$, where $\partial W=M(=\partial_1 W)\cup\partial_2 W\cup \partial_3 W\cup \partial_4 W$. Moreover, we have decompositions $\partial M=\partial_1 M\cup \partial_2 M\cup \partial_3 M$, $\partial(\partial_2 W)=\partial\partial_{2, 1} W\cup \partial\partial_{2, 3} W\cup \partial\partial_{2, 4} W$, $\partial(\partial_3 W)=\partial\partial_{3, 1} W\cup \partial\partial_{3, 2} W\cup \partial\partial_{3, 4} W$,  and $\partial(\partial_4 W)=\partial\partial_{4, 1} W\cup \partial\partial_{4, 2} W\cup \partial\partial_{4, 3} W$ such that
			$$
			\partial_1M=\partial\partial_{1, 2} W, \ \partial_2M=\partial\partial_{1, 3} W, \ \ and \ \ \partial_3M=\partial\partial_{1, 4} W
			$$
			and
			$$
			\partial\partial_{i, j}W=\partial\partial_{j, i} W=\partial_i W\cap \partial_j W \ \ for \ any \ i, j=1,2,3,4.
			$$
			Furthermore, we have
			\begin{align*}
			\partial_1M\cap \partial_2M&=\partial\partial_{1, 2} W\cap \partial\partial_{1, 3} W=\partial\partial_{2, 1} W\cap \partial\partial_{2, 3} W=\partial\partial_{3, 1} W\cap \partial\partial_{3, 2} W \\
			&=\partial_{1} W\cap \partial_{2} W \cap \partial_{3} W=\partial\partial\partial_{1, 2, 3} W,
			\end{align*}
			\begin{align*}
			\partial_1M\cap \partial_3M&=\partial\partial_{1, 2} W\cap \partial\partial_{1, 4} W=\partial\partial_{2, 1} W\cap \partial\partial_{2, 4} W=\partial\partial_{4, 1} W\cap \partial\partial_{4, 2} W \\
			&=\partial_{1} W\cap \partial_{2} W \cap \partial_{4} W=\partial\partial\partial_{1, 2, 4} W
			\end{align*}
			\begin{align*}
			\partial_2M\cap \partial_3M&=\partial\partial_{1, 3} W\cap \partial\partial_{1, 4} W=\partial\partial_{3, 1} W\cap \partial\partial_{3, 4} W=\partial\partial_{4, 1} W\cap \partial\partial_{4, 3} W \\
			&=\partial_{1} W\cap \partial_{3} W \cap \partial_{4} W=\partial\partial\partial_{1, 3, 4} W,
			\end{align*}
			and
			\begin{align*}
			\partial_1M\cap \partial_2M\cap \partial_3M&=\partial\partial_{1, 2} W\cap \partial\partial_{1, 3} W \cap \partial\partial_{1, 4} W \\
			&=\partial\partial_{2, 1} W\cap \partial\partial_{2, 3} W \cap \partial\partial_{2, 4} W \\
			&=\partial\partial_{3, 1} W\cap \partial\partial_{3, 2} W \cap \partial\partial_{3, 4} W \\
			&=\partial\partial_{4, 1} W\cap \partial\partial_{4, 2} W \cap \partial\partial_{4, 3} W \\
			&=\partial_{1} W\cap \partial_{2} W \cap \partial_{3} W \cap \partial_{4} W \\
			&=\partial\partial\partial\partial_{1,2, 3,4} W.
			\end{align*}
			\item Similarly, we have a manifold 4-ads $(V, \partial V)$ of dimension $(n+1)$ with a continuous map $\Psi: (V, \partial_4 V)\rightarrow (X, \partial X)$ so that $\Psi^*(\omega)$ describes the orientation character of $V$, where $\partial V=N(=\partial_1 V)\cup\partial_2 V\cup \partial_3 V\cup\partial_4 V$ satisfying similar conditions as $W$.
			
			\item There is a degree one normal map of manifold 4-ads $F: (V, \partial V) \rightarrow (W, \partial W)$ such that  $\Phi\circ F=\Psi$. Moreover, $F$ restricts to $f$ on $N\subseteq \partial V$.
			\item The restriction $F|_{\partial_k V}: \partial_k V\to \partial_k W$ is a degree one normal map over $X$ for $k=1,2,4$.
			\item The restriction $F|_{\partial_3 V}: \partial_3 V\to \partial_3 W$ is a homotopy equivalence over $X$ and it restricts to an infinitesimally controlled homotopy equivalence $F|_{\partial\partial_{2,3} V}: \partial\partial_{2,3} V\to \partial\partial_{2,3} W$ over $X$.
		\end{enumerate}
		
	\end{definition}
	
	Let $L_n(\pi_1X, \pi_1(\partial X), X; \omega)$ be the set of equivalence classes from Definition \ref{equi rel LX}.  By definition, one can see that $L_n(\pi_1X, \pi_1(\partial X), X; \omega)$ is actually a group with the sum operation being disjoint union.

	Now let us form our description of the relative topological surgery exact sequence.

	Note that there is a natural group homomorphism
	$$
	i_{*}: N_n(X, \partial X; \omega)\rightarrow L_n(\pi_1X, \pi_1(\partial X); \omega)
	$$
	by forgetting control.

	Define
	$$
	j_{*}: L_n(\pi_1X, \pi_1(\partial X); \omega)\rightarrow L_n(\pi_1X, \pi_1(\partial X), X; \omega)
	$$
	by
	$$
	j_{*}(\theta)=\{M, (\emptyset, \partial_{+} M, \partial_{-} M), \phi, N, (\emptyset, \partial_{+} N, \partial_{-} N), \psi, f\}
	$$
    for $\theta=\{M, \partial_{\pm} M, \phi, N, \partial_{\pm} N, \psi, f\}$, and define
	$$
	\partial_{*}: L_{n+1}(\pi_1X, \pi_1(\partial X), X; \omega)\rightarrow N_n(X, \partial X; \omega)
	$$
	by
	$$
	\partial_{*}(\theta)=\partial_{1}(\theta)=\theta_1=\{\partial_1M, (\partial\partial_{1,2} M, \partial\partial_{1,3} M), \phi, \partial_1N, (\partial\partial_{1,2} M, \partial\partial_{1,3} M), \psi, f\}
	$$
for any $\theta=\{M, \partial_{k} M, \phi, N, \partial_{k} N, \psi, f; k=1,2,3.\}$.
	Furthermore, we call $\theta_1$ the $\partial_1$-boundary of $\theta$ and we may define $\partial_k$-boundary similarly.

	\begin{theorem}\label{Les1}
		We have the following long exact sequence		
		\begin{align*}
		\cdots &\stackrel{}{\longrightarrow} L_{n+1}(\pi_1X, \pi_1(\partial X),X; \omega)  \stackrel{\partial_*}{\longrightarrow} N_{n}(X, \partial X; \omega) \stackrel{i_*}{\longrightarrow} L_{n}(\pi_1X, \pi_1(\partial X); \omega) \\
		&\stackrel{j_*}{\longrightarrow} L_{n}(\pi_1X, \pi_1(\partial X),X; \omega)  \stackrel{\partial_*}{\longrightarrow} N_{n-1}(X, \partial X; \omega) \stackrel{}{\longrightarrow} \cdots
		\end{align*}		
	\end{theorem}
	\begin{proof}
		(I) {\bf Exactness at $N_{n}(\pi_1X, \pi_1(\partial X); \omega)$.} Let $\theta\in N_{n}(\pi_1X, \pi_1(\partial X); \omega)$.
		Then $i_*(\theta)=0$ if and only if there exists an element
		$$
		\eta=\{W, \partial_{k} W, \Phi, V, \partial_{k} V, \Psi, F; k=1,2,3.\}
		$$
		satisfying the conditions in \ref{equi rel L}. Note that $\eta$ is an element in $L_{n+1}(\pi_1X, \pi_1(\partial X),X; \omega)$ and is mapped to $\theta$ under $\partial_*$. This proves the exactness at $N_{n}(\pi_1X, \pi_1(\partial X); \omega)$.

		(II) {\bf Exactness at $L_{n}(\pi_1X, \pi_1(\partial X); \omega)$.} Let
		$$
		\xi=\{M, \partial_{\pm} M, \phi, N, \partial_{\pm} N, \psi, f\}\in N_{n}(\pi_1X, \pi_1(\partial X); \omega).
		$$
		Then $j_*i_*(\xi)=0$ since $\xi\times I$ is a cobordism of $\xi$ to the empty set where $I$ is the unit interval. More precisely, $\xi\times I$ consists of the following data.
		
		(i) $W=M\times I$ with continuous map
		$$
		\Phi=\phi\circ p_1: (W, \partial_4W)\stackrel{p_1}{\rightarrow} (M, \partial_{-}M) \stackrel{\phi}{\rightarrow} (X, \partial X),
		$$
		where $p_1: W\rightarrow M$ is the natural projection, $\partial W=\partial_1 W(=M\times\{0\})\cup\partial_2 W\cup \partial_3 W\cup \partial_4 W$ with $\partial_2 W=M\times\{1\}$, $\partial_3 W=\partial_{+}M\times I$ and $\partial_4 W=\partial_{-}M\times I$.
		
		(ii) There is a similar picture for $(V, \partial V)$ with $\partial V=\partial_1 V(=N \times\{0\})\cup\partial_2 V\cup \partial_3 V\cup \partial_4 V$, where $\partial_2 V=N\times\{1\}$, $\partial_3 V=\partial_{+}N\times I$ and $\partial_4 V=\partial_{-}N\times I$.

		(iii) A degree one normal map of manifold 4-ads, $F=f\times Id : (V, \partial V)\rightarrow (W, \partial W)$. Obviously, $\Phi\circ F=\Psi$ and $F$ restricts to $f$ on $N\subseteq \partial V$.
		
		(iv) $F|_{\partial_3 V}: \partial_3 V=\partial_{+}N\times I\rightarrow \partial_3 W=\partial_{+}M\times I$ is a homotopy equivalence. This is because $f: \partial_{+}N\rightarrow \partial_{+}M$ is an infinitesimally controlled homotopoy equivalence.
		
		(v) Moreover, $F|_{\partial\partial_{2,3} V}: \partial\partial_{2,3} V=\partial_{+}N\to \partial\partial_{2,3} W=\partial_{+}M$ is  an infinitesimally controlled homotopy equivalence over $X$.

		Conversely, suppose an element
		$$
		\theta=\{M, \partial_{\pm} M, \phi, N, \partial_{\pm} N, \psi, f\}\in L_{n}(\pi_1X, \pi_1(\partial X); \omega)
		$$
		is mapped to zero in $L_{n}(\pi_1X, \pi_1(\partial X), X; \omega)$. Then
		$$
		j_{*}(\theta)=\{M, (\emptyset, \partial_{+} M, \partial_{-} M), \phi, N, (\emptyset, \partial_{+} N, \partial_{-} N), \psi, f\}
		$$
		is cobordant to empty set in $L_{n}(\pi_1X, \pi_1(\partial X), X; \omega)$. More precisely, we have the following data:
		
		\begin{enumerate}
			\item There exists a manifold 4-ads $(W, \partial W)$ of dimension $(n+1)$ with a continuous map $\Phi: (W, \partial_4 W)\rightarrow (X, \partial X)$ so that $\Phi^*(\omega)$ describes the orientation character of $W$, where $\partial W=M(=\partial_1 W)\cup\partial_2 W\cup \partial_3 W\cup \partial_4 W$.
			\item We have decompositions $\partial M=\partial_1 M(=\emptyset)\cup \partial_2 M(=\partial_+ M)\cup \partial_3 M(=\partial_- M)$, $\partial(\partial_2 W)=\partial\partial_{2, 1} W\cup \partial\partial_{2, 3} W\cup \partial\partial_{2, 4} W$, $\partial(\partial_3 W)=\partial\partial_{3, 1} W\cup \partial\partial_{3, 2} W\cup \partial\partial_{3, 4} W$,  and $\partial(\partial_4 W)=\partial\partial_{4, 1} W\cup \partial\partial_{4, 2} W\cup \partial\partial_{4, 3} W$ such that
			$$
			\partial_1M=\emptyset=\partial\partial_{1, 2} W, \ \partial_2M=\partial_+ M=\partial\partial_{1, 3} W, \ \text{and} \ \ \partial_3M=\partial_- M=\partial\partial_{1, 4} W.
			$$
			Moreover, we have $\partial\partial_{1, 3} W\cap\partial\partial_{2, 3} W=\emptyset$.
			\item Similarly, we have a manifold 4-ads $(V, \partial V)$ of dimension $(n+1)$ with a continuous map $\Psi: (V, \partial_4 V)\rightarrow (X, \partial X)$ so that $\Psi^*(\omega)$ describes the orientation character of $V$, where $\partial V=N(=\partial_1 V)\cup\partial_2 V\cup \partial_3 V\cup\partial_4 V$ satisfying similar conditions as $W$.
			
			\item There is a degree one normal map of manifold 4-ads $F: (V, \partial V) \rightarrow (W, \partial W)$ such that  $\Phi\circ F=\Psi$. Moreover, $F$ restricts to $f$ on $N\subseteq \partial V$.
			\item The restriction $F|_{\partial_k V}: \partial_k V\to \partial_k W$ is a degree one normal map over $X$ for $k=1,2,4$.
			\item The restriction $F|_{\partial_3 V}: \partial_3 V\to \partial_3 W$ is a homotopy equivalence over $X$ and it restricts to an infinitesimally controlled homotopy equivalence $F|_{\partial\partial_{2,3} V}: \partial\partial_{2,3} V\to \partial\partial_{2,3} W$ over $X$.
		\end{enumerate}
		Consequently, $F: (V, \partial V) \rightarrow (W, \partial W)$ provides a cobordism between $\theta$ and
		$$
		\eta=\{\partial_3W, (\partial\partial_{2,3} W, \partial\partial_{3,4} W), \Phi|_{\partial_3W}, \partial_3V, (\partial\partial_{2,3} V, \partial\partial_{3,4} V), \Psi|_{\partial_3V}, F\}.
		$$
		Note that $\eta$ is an element in $N_{n}(\pi_1X, \pi_1(\partial X); \omega)$. This prove the exactness at $L_{n}(\pi_1X, \pi_1(\partial X); \omega)$.
		
		(III) {\bf Exactness at $L_{n}(\pi_1X, \pi_1(\partial X),X; \omega)$.} It is obvious that $\partial_*j_*=0$ by definition. On the other hand, if an element
		$$
		\theta=\{M, \partial_{k} M, \phi, N, \partial_{k} N, \psi, f; k=1,2,3.\}\in L_{n}(\pi_1X, \pi_1(\partial X),X; \omega)
		$$
		such that $\partial_*(\theta)=0$, then there is a cobordism of $\partial_*(\theta)$ to the empty set, i.e.
		$$
		\eta=\{W, \partial_{k} W, \Phi, V, \partial_{k} V, \Psi, F; k=1,2,3.\}
		$$
		following from Definition \ref{equi rel N}. Consequently, Let $\theta'=\eta\cup_{\partial_*(\theta)}\theta$. Then a cobordism of $\theta'$ to $\theta$ is provided by $\theta'\times I$ with $\partial_{1} (\theta'\times I)=\theta'\times \{0\}\cup \theta\times \{1\}$, $\partial_{2} (\theta'\times I)=\eta\times \{1\}$, $\partial_{3} (\theta'\times I)=\partial_2\theta'\times I$ and $\partial_{4} (\theta'\times I)=\partial_3\theta'\times I$. Note that the $\partial_1$-boundary of $\theta'$ is empty, so $\theta'$ is the image of $j_*$ of some element in $L_{n}(\pi_1X, \pi_1(\partial X); \omega)$. This proves the exactness at $L_{n}(\pi_1X, \pi_1(\partial X),X; \omega)$.
	\end{proof}
	
	There is a natural group homomorphism
	$$
	c_{*}: S_n(\pi_1X, \pi_1(\partial X); \omega)\rightarrow L_{n+1}(\pi_1X, \pi_1(\partial X), X; \omega)
	$$
	by mapping
	$$
	\theta=\{M, \partial_{\pm} M, \phi, N, \partial_{\pm} N, \psi, f\}\mapsto\theta\times I
	$$
	where $\theta\times I$ consists of the following data:
	
	(1) a manifold 3-ad $(M\times I, \partial_k(M\times I); k=1,2,3)$ with $\partial_1(M\times I)=(M\times \{0\})\cup_{\partial_+M\times \{0\}}(\partial_+M\times I)$, $\partial_2(M\times I)=M\times \{1\}$ and $\partial_3(M\times I)=\partial_-M\times I$; in particular, $\partial\partial_{1,2}(M\times I)=\partial_+M$;
	
	(2) similarly, another manifold 3-ad $(N\times I, \partial_k(N\times I); k=1,2,3)$ with $\partial_1(N\times I)=(N\times \{0\})\cup_{\partial_+N\times \{0\}}(\partial_+N\times I)$, $\partial_2(N\times I)=N\times \{1\}$ and $\partial_3(N\times I)=\partial_-N\times I$;
	
	(3) a continuous map
	$$
	\widetilde{\phi}:=\phi\circ p_1: (M\times I, \partial_3(M\times I)) \stackrel{p_1}{\rightarrow} (M, \partial_-M) \stackrel{\phi}{\rightarrow} (X, \partial X)
	$$
	such that $(\phi\circ p_1)^{*}(\omega)$ describes the orientation character of $M\times I$, where $p_1$ is the canonical projection map from $M\times I$ to $M$; similarly, a continuous map
	$$
	\widetilde{\psi}:=\phi\circ p_2: (N\times I, \partial_3(N\times I)) \stackrel{p_2}{\rightarrow} (N, \partial_-N) \stackrel{\psi}{\rightarrow} (X, \partial X)
	$$
	describes the orientation character of $N\times I$, where $p_2$ is the canonical projection map from $N\times I$ to $N$;
	
	(4) a degree one normal map of manifold 3-ads
	$$
	\widetilde{f}:=f\times Id : (N\times I, \partial_k(N\times I); k=1,2,3)\rightarrow (M\times I, \partial_k(M\times I); k=1,2,3)
	$$
	such that $\widetilde{\phi}\circ\widetilde{f}=\widetilde{\psi}$;
	
	(5) the restriction $\widetilde{f}|_{\partial_1 (N\times I)}: \partial_1 (N\times I)\to \partial_1 (M\times I)$ is a degree one normal map (homotopy equivalence) over $X$;
	
	(6) the restriction $\widetilde{f}|_{\partial_2 (N\times I)}: \partial_2 (N\times I)\to \partial_2 (M\times I)$ is a homotopy equivalence over $X$ and it restricts to an infinitesimally controlled homotopy equivalence $\widetilde{f}|_{\partial\partial_{1,2} (N\times I)}: \partial\partial_{1,2} (N\times I)\to \partial\partial_{1,2} (M\times I)$ over $X$;
	
	(7) the restriction $\widetilde{f}|_{\partial_3 (N\times I)}: \partial_3 (N\times I)\to \partial_3 (M\times I)$ is a degree one normal map over $X$.

	Define
	$$
	r_{*}: L_{n+1}(\pi_1X, \pi_1(\partial X), X; \omega)\rightarrow S_n(\pi_1X, \pi_1(\partial X); \omega)
	$$
	by
	$$
	r_{*}(\theta)=\partial_2(\theta)=\theta_2=\{\partial_2 M, (\partial\partial_{1,2} M, \partial\partial_{2,3} M), \phi, \partial_2N, (\partial\partial_{1,2} N, \partial\partial_{2,3} N), \psi, f\},
	$$
    for $\theta=\{M, \partial_{k} M, \phi, N, \partial_{k} N, \psi, f; k=1,2,3.\}$,
	where $\partial\partial_{1,2} M$ means $\partial_+(\partial_2M)$ and $\partial\partial_{2,3} M$ means $\partial_-(\partial_2M)$ (resp. for $N$).
	
	\begin{theorem}
		The homomorphisms $c_*$ and $r_*$ are inverse of each other. In particular, we have $S_n(\pi_1X, \pi_1(\partial X); \omega)\cong L_{n+1}(\pi_1X, \pi_1(\partial X), X; \omega)$.
	\end{theorem}
	\begin{proof}
		First, it is obvious that
		$$
		r_*\circ c_*=Id: S_n(\pi_1X, \pi_1(\partial X); \omega)\rightarrow S_n(\pi_1X, \pi_1(\partial X); \omega).
		$$
		
		Conversely, for any
		$$
		\theta=\{M, \partial_{k} M, \phi, N, \partial_{k} N, \psi, f; k=1,2,3.\}\in L_{n+1}(\pi_1X, \pi_1(\partial X), X; \omega),
		$$
		$c_*r_*(\theta)$ is cobordant to $\theta$ in $L_{n+1}(\pi_1X, \pi_1(\partial X), X; \omega)$. Indeed, Consider the element
		$$
		(\theta\times I)\bigcup\limits_{(\theta_2\times I)\times \{0\}\subseteq\theta\times \{1\}}(\theta_2\times I\times I)
		$$
		where $(\theta_2\times I)\times \{0\}$ is glued to the subset $(\theta_2\times I)\subseteq\theta$ in $\theta\times \{1\}$.
		This produces a cobordism between $c_*r_*(\theta)$ and $\theta$, which completes the proof.
	\end{proof}
	
Put $\widetilde{\partial}_*=\partial_*\circ c_*$. We could replace $L_{n+1}(\pi_1X, \pi_1(\partial X),X; \omega)$ and ${\partial}_*$ by $S_{n}(X, \partial X; \omega)$ and $\widetilde{\partial}_*$ in the long exact sequence in Theorem \ref{Les1}, respectively.

	\begin{theorem}\label{Les}
		We have the following long exact sequence
	\begin{align*}
		\cdots &\stackrel{}{\longrightarrow} S_{n}(X, \partial X; \omega)  \stackrel{\widetilde{\partial}_*}{\longrightarrow} N_{n}(X, \partial X; \omega) \stackrel{i_*}{\longrightarrow} L_{n}(\pi_1X, \pi_1(\partial X); \omega) \\
		&\stackrel{j_*}{\longrightarrow} S_{n-1}(X, \partial X; \omega)  \stackrel{\widetilde{\partial}_*}{\longrightarrow} N_{n-1}(X, \partial X; \omega) \stackrel{}{\longrightarrow} \cdots.
		\end{align*}
	\end{theorem}

	\section{Geometric $C^*$-algebras}\label{sec K theory preparation}
	
	In this section, we introduce the definition of  the relative  equivariant maximal Roe algebra in light of \cite{CWY15}. We shall start with the definition of the equivariant maximal Roe algebra.

 All manifolds and manifolds with boundary considered in the following are oriented.

	\subsection{Maximal Roe algebra}\label{subsec basic notions}
	We first recall the definition of the maximal Roe algebra.

Let $X$ be a proper metric space with bounded geometry. Let $G$ be a discrete group acting freely, cocompactly and properly on $X$. A $G$-equivariant $X$ module $H_{X}$ is a separable Hilbert space equipped with a $*$-representation $\phi$ of $C_0(X)$ and a covariant $G$ action $\pi$ such that
	\[
	\pi(g)(\phi(f)v)=\phi(f^g)(\pi(g)(v)), \ \  \forall g\in G, f\in C_0(X) \ \text{and} \ v\in H_{X},
	\]
	where $f^g(x)=f(g^{-1}x).$ We call $H_{X}$ standard if no nonzero function in $C_0(X)$ acts as a compact operator,  non-degenerate if the $*$-representation $\phi$ of $C_0(X)$ is non-degenerate.
	
	\begin{definition}[cf. \cite{Roe93}]\label{def all three factor}
		Let $H_{X}$ be a $G$-equivariant, standard, and non-degenerate $X$-module.
		\begin{enumerate}
			\item The support $\text{supp}(T)$ of a bounded linear operator $T\in B(H_{X})$ is defined to be the complement of the set of all points $(x,y)\in X\times X$ for which there exist $f,\ g \in C_0(X)$ such that $gTf=0$, $f(x)\neq 0$, $g(y)\neq 0$.
			\item A bounded linear operator $T\in B(H_{X})$ is said to have finite propagation if
			\[
			\text{sup}\{d(x,y): (x,y)\in \text{Supp}(T)\}< \infty.
			\]
			This number will be called the propagation of $T$, and denoted as $\text{propagation}(T)$.
			\item A bounded linear operator $T\in B_{X}$ is said to be locally compact if  $fT$ and $Tf$ are both compact operators for all $f\in C_0(X)$.
		\end{enumerate}
	\end{definition}
	
	Denote by $C[X]^G$ the set of all locally compact, finite propagation $G$-invariant operators on $H_{X}$.
	
	\begin{definition}\label{def of three operator algebra}
		Let $X$ be a proper metric space with bounded geometry. The discrete group $G$ acts on $X$ freely, cocompactly, and properly. Then The maximal Roe algebra $C^*_{max}(X)^G$ is the completion of $C[X]^G$ with respect to the $C^*$-norm
			\[\|T\|_{max} := \text{sup}\{\|\psi(T)\|_{B(H_\psi)}|\  \psi: C[X]^G \to B(H_\psi),\ \text{a } *-\text{representation}\}.\]
			In fact, we have that $C^*_{max}(X)^G\cong C^*_{max}(G)\otimes \mathcal{K},$ where $\mathcal{K}$ is the $C^*$-algebra consists of compact operators.
			%\item The maximal localization algebra $C^*_{L,max}(X)^G$ is the $C^*$-algebra generated by all  bounded and uniformly norm-continuous functions $f: [0,\infty)\to C^*_{max}(X)^G$ such that $
			%\text{the propagation of }f(t)\to 0,\  \text{as } t\to \infty.
			%$
		   %\item The maximal obstruction algebra $C^*_{L,0,max}(X)^G$ is the kernel of the following evaluation map
			%\[
			%\text{ev}: C^*_{L,max}(X)^G\to C^*_{max}(X)^G,\  \text{ev}(f)=f(0).
			%\]
			%\item If $Y$ is a subspace of $X$ and $G$ acts on $Y$ freely and properly, then $C^*_{L, max}(Y, X)^G$ (resp. $C^*_{L,0,\max}(Y;X)^G$) is defined to be the $C^*$-subalgebra of $C^*_{L, max}(X)^G$ (resp. $C^*_{L,0, max}(X)^G$) generated by all elements $f$ such that there exist $c_t>0$ satisfying $\lim_{t \to \infty}c_t=0$ and $$\text{Supp}(f(t))\subset \{(x,y)\in X\times X|d((x,y), Y\times Y)\leq c_t\} $$ for all $t$.
		%\end{enumerate}
	\end{definition}

	\subsection{Relative Roe algebra}\label{subsec relative C alg}
	
	In this subsection, we recall the definition of the relative Roe algebra  in light of \cite{CWY15}.

We start with the following construction.

	\begin{definition}
		Let $\iota: A\to B$ be a $C^*$-algebra homomorphism. We define $C_{\iota: A\to B}$ to be the $C^*$-algebra generated by
		\[
		\{(a,f) : f\in C_0([0,1), B), a\in A, f(0)=\iota(a)  \}.
		\]
	\end{definition}

	For a manifold with boundary $(M, \partial M)$, let $p:\widetilde{M}\to M$  and $p': \widetilde{\partial M}\to \partial M$  be the universal covering maps of $M$ and $\partial M$ respectively, and let $\widetilde{\partial M}'$ be $p^{-1} \partial M$ .
%     Let
%	\[
%	i: \widetilde{\partial M}\to \widetilde {\partial M}'\hookrightarrow \widetilde M
%	\]
%	be the embedding map and
Let 	
$$
	j: \pi_1(\partial M)\to \pi_1(M)
	$$
	be the homomorphism induced by the inclusion of the boundary. Let $\widetilde{\partial M}''$ be the Galois covering space of $\partial M$ whose Deck transformation group is $j\pi_1(\partial M)$. We have $\widetilde{\partial M}'= \pi_1 (M) \times_{j\pi_1 (\partial M)} \widetilde{\partial M}'' $. This decomposition naturally gives rise to a homeomorphism
\begin{equation}\label{eq iota of manifold}
\iota: \widetilde{\partial M}\to \widetilde {\partial M}'\hookrightarrow \widetilde M
\end{equation}
and a  $*$-homomorphism
	\[
	\phi': C^*_{max} (\widetilde{\partial M}'')^{j\pi_1 (\partial M)} \to C^*_{max}(\widetilde{\partial M}')^{\pi_1(M)}\hookrightarrow C^*_{max}(\widetilde{M})^{\pi_1( M)} .
	\]
	Lemma 2. 12 of \cite{CWY15} shows that there is a natural $*$-homomorphism
	\[
	\phi'': C^*_{max} (\widetilde{\partial M})^{\pi_1 (\partial M)} \to  C^*_{max} (\widetilde{\partial M}'')^{j\pi_1 (\partial M)}.
	\]
	Thus
	\[
	\phi'\phi'' : C^*_{max}(\widetilde{\partial M})^{\pi_1(\partial M)}\to C^*_{max}(\widetilde{M})^{\pi_1( M)}
	\]
is  a $C^*$-algebra homomorphism, which will be denoted by $\iota$ with a little abuse of notation.
	
	%Similarly, one can see that $i: \partial M\to M$ also induces the following two $*$-homomorphisms
%	\begin{eqnarray*}
%		i_L &:& C^*_{L,max}(\widetilde{\partial M})^{\pi_1(\partial M)}\to C^*_{L,max}(\widetilde{M})^{\pi_1( M)},\\
%		i_{L,0} &:& C^*_{L,0,max}(\widetilde{\partial M})^{\pi_1(\partial M)}\to C^*_{L,0,max}(\widetilde{M})^{\pi_1( M)}.
%	\end{eqnarray*}
	
	For any $C^*$-algebra $A$, let $SA$ be its suspension algebra.
	\begin{definition}[Relative maximal algebras]\label{def rel max roe}
		For a manifold with boundary $(M, \partial M)$,
			the relative maximal Roe algebra associated to it is then defined as
			\[
			C_{max}^*(\widetilde{M}, \widetilde{\partial M})^{\pi_1(M), \pi_1(\partial M)} := SC_{\iota}.
			\]
			
	\end{definition}
	%Let $[v]$ be a generator of $K_1(C(S^1))$. An element in
	Since all the Roe algebras considered in this paper are maximal ones, we oppress the subscription $max$ in the following. The relative algebras defined above are then denoted by $C^*(\widetilde{M}, \widetilde{\partial M})^{\pi_1(M), \pi_1(\partial M)}$.  No confusion should be arose.
	
%	Let $G$ be $\pi_1(M)$ and $\Gamma$ be $\pi_1(\partial M)$.
%	The following $K$-theory six exact sequence is routine:
%	\[
%	{\small
%		\xymatrix{
%			K_0 (C^*_{L, 0}(\widetilde{M}, \widetilde{\partial M})^{G,\Gamma})\ar[r] &K_0 (C^*_{L}(\widetilde{M}, \widetilde{\partial M})^{G,\Gamma})\ar[r]^{\text{ev}} & K_0 (C^*(\widetilde{M}, \widetilde{\partial M})^{G, \Gamma})  \ar[d]^{\partial}\\
%			K_1 (C^*(\widetilde{M}, \widetilde{\partial M})^{G,\Gamma}) \ar[u]^{\partial} & K_1 (C^*_{L}(\widetilde{M}, \widetilde{\partial M})^{G,\Gamma})\ar[l]_{\text{ev}} & K_1 (C^*_{L, 0}(\widetilde{M}, \widetilde{\partial M})^{G,\Gamma})\ar[l]
%		}
%	}.\]

	\section{Signature of compact PL manifolds}\label{sec relative index}
	
	In this section we recall the definition of the  signature of compact PL manifolds.  %Generally, one can define the signature  for a geometrically controlled Hilbert-Poincar\'e complex (cf. \cite{HR051}, \cite{HR052}). A triangulation of a PL manifold $X$ gives rise to a geometrically controlled Hilbert-Poincar\'e complex over $X$. In this particular case, one can further define the $K$-homology class of the signature operator (cf. \cite{WXY18}). At last, we recite the definition of the higher $\rho$ invariant for a homotopy equivalence between two PL manifolds.
The readers are referred to \cite{HR051}, \cite{HR052} and \cite{WXY18} for more details.

\subsection{Analytically controlled Hilbert-Poincar\'e complex}

In this subsection, we recall the definition of the analytically controlled Hilbert-Poincar\'e complex. We first introduce the definition of the analytically controlled operator.

Let $X$ be a proper metric space with bounded geometry and $G$ be a discrete group acting freely, cocompactly,  and properly on $X.$

\begin{definition}
Let $H_0$ and $H_1$ be two $G$-equivariant $X$-module. A bounded operator $T:H_0\to H_1$ is said to be $G$-equivariant  analytically controlled over $X$ if it is the norm limit of  $G$-equivariant, locally compact and finite propagation bounded operators.
\end{definition}

Now we define the $G$-equivariant analytically controlled complex.

\begin{definition}
A chain complex
\[
(H_{*,X}, b)^G: H_{n,X}\stackrel{b}{\rightarrow }  H_{n-1, X}\stackrel{b}{\rightarrow } \cdots \stackrel{b}{\rightarrow } H_{1,X}\stackrel{b}{\rightarrow } H_{0, X},\]
is called an $n$-dimensional $G$-equivariant  analytically controlled Hilbert complex over $X$ if each $H_p$ is $X$-module and each $b$ is $G$-equiavariant  analytically controlled over $X$.
\end{definition}

Now let us recall the definition of the $G$-equivariant analytically controlled  chain homotopy equivalence between $G$-equivariant analytically controlled Hilbert complexes.

\begin{definition}\label{def ana control homotopy of hil complex bounded}
A chain homotopy equivalence
\[
A: (H_{*, X}, b)^G \to (H'_{*,X}, b')^G
\]
 between $G$-equivariant analytically controlled Hilbert complexes over $X$, is said to be $G$-equivariant analytically controlled over $X$ if
\begin{enumerate}
\item $A$ is $G$-equivariant analytically controlled over $X,$
\item there exist $G$-equivariant analytically controlled  chain maps
\[B: (H'_{*,X}, b')^G \to (H_{*,X}, b)^G, \]
and $G$-equivariant analytically controlled operators $y, \ y'$ with degree $1$, i.e.
\[
y: H_{i, \widetilde{N}} \to H_{i+1, \widetilde{N}}, \ y': H'_{i, \widetilde{N}} \to H'_{i+1, \widetilde{N}},
\]
such that
    \[
    I-AB= b'y'+y'b', I-BA = by+yb.
    \]
\end{enumerate}
\end{definition}

The analytically controlled Hilbert-Poincar\'e complex is an analytically controlled Hilbert complex equipped with the Poincar\'e duality.

\begin{definition}\label{def hilbert Poincare complex bounded}
 A $G$-equivariant analytically controlled Hilbert-Poincar\'e complex over $X$, denoted as $(H_{*,X}, b, T)^G,$ is a $G$-equivariant analytically controlled Hilbert complex over $X$
 \[
 (H_{*,X}, b)^G: H_{n,X}\stackrel{b}{\rightarrow }  H_{n-1, X}\stackrel{b}{\rightarrow } \cdots \stackrel{b}{\rightarrow } H_{1,X}\stackrel{b}{\rightarrow } H_{0, X},\]
 equipped  with adjointable bounded operator $T: H_{*, X}\to H_{n-*, X},$ such that
 \begin{enumerate}
 \item  $T^*(v)=(-1)^{(n-p)p}T(v),$ if $v\in  H_{p, X},$
   \item $Tb^*(v) + (-1)^pbT(v)=0$,  if $v\in  H_{p, X},$
   \item $T$ is a $G$-equivariant analytically controlled chain homotopy equivalence over $X$ from the dual complex
       \[(H_{n-*,X}, b^*)^G: H_{0, X}\stackrel{b^*}{\rightarrow }  H_{1, X}\stackrel{b^*}{\rightarrow } \cdots \stackrel{b^*}{\rightarrow } H_{n-1,X}\stackrel{b^*}{\rightarrow } H_{n, X}\]
to  $(H_{*,X}, b)^G.$

In the following, we will call $T$ the Poincar\'e duality operator of $(H_{*,X}, b)^G.$
 \end{enumerate}

\end{definition}

We mention that one need appropriate  signs to make $T$ into a genuine chain map, however for the sake of conciseness, we leave it as is. The reader should not be confused.

Correspondingly, we have the following notion of the $G$-equivariant analytically controlled homotopy equivalence between Hilbert-Poincar\'e complexes.

\begin{definition}\label{def ana control homotopy of hil poinc complex bounded}
Let $(H_{*, X}, b, T)^G$ and $(H'_{*, X}, b', T')^G$ be two $G$-equivariant analytically controlled Hilbert-Poincar\'e complexes over $X$.
Let
\[
A: (H_{*,X}, b)^G \to (H'_{*,X}, b')^G
\]
be a  $G$-equivariant analytically controlled chain homotopy equivalence. Then the homotopy equivalence  $A$ is said to be $G$-equivariant analytically controlled chain homotopy equivalence between  $(H_{*,X}, b, T)^G$ and $(H'_{*,X}, b', T')^G$, if
\[
T', ATA^* :   (H'_{n-*,X}, (b')^*)^G\to (H'_{*,X}, b')^G.
\]
are analytically controlled homotopy equivalent to each other,  i.e.  there exist $G$-equivariant analytically controlled operators $y: H_{*,X} \to H_{n-*-1,X}$, such that
\[
ATA^*-T'=yb^*+by .
\]

In the following, T is called the duality operator of the controlled Hilbert-Poincar\'e complex $(H_{*, X}, b, T)^G .$

\end{definition}

\subsection{Signature of Hilbert-Poincar\'e complexes}

In this subsection, we recall the definition of the signature of $G$-equivariant analytically controlled  Hilbert-Poincar\'e complexes.

\begin{definition}\label{chiral duality}
Let $(H_{*, X}, b, T)^G$ be an $n$-dimensional $G$-equivariant analytically controlled Hilbert-Poincar\'e complex over $X,$ let $l$ be $[\frac{n}{2}].$ Set $\gamma=i^{p(p-1)+l}, p=0,1,\cdots, n.$ Define the chirality duality operator $S:H_{*, X}\to H_{n-*, X}$ to be the bounded self-adjoint operator such that
\[
S(v)=\gamma T(v), \forall v\in H_{p,X}.
\]

\end{definition}

It is straightforward to verify that $S=S^*,$ and that $bS+Sb^*=0.$ In \cite{HR051}, Higson and Roe proved that  both of $b +b^*\pm S$ are self-adjoint invertible operators (\cite{HR051}). Set $B:= b +b^*$.
The following is the definition of the signature of 	 $(H_{*,X}, b, T)^G$:
	\begin{definition}\label{def signature of hp complex}
		\begin{enumerate}
			\item Let $(H_{*,X}, b, T)^G$ be an odd dimensional $G$-equivariant analytically controlled Hilbert-Poincar\'e complex over $X$. It was shown in \cite{HR051} that the following operator
			\[
			\frac{B+S}{B-S}: H_{ev,X}\to H_{ev,X}
			\]
belongs to $(C^*(X)^G)^+$, where $H_{ev,X}$ equals $\oplus_{k}H_{2k,X}$. The signature of $(H_{*,X}, b, T)^G$ is then defined to be the $K_1(C^*(X)^G)$ class represented by
\[\frac{B+S}{B-S}: H_{ev,X}\to H_{ev,X}.\]

			\item  Let $(H_{*,X},b, T)^G$ be an even dimensional $G$-equivariant analytically controlled Hilbert-Poincar\'e complex over $X$. It was shown in \cite{HR051} that $P_+(B\pm S)$, the positive spectral projection of $B\pm S$ can be approximated by finite propagation operators, and that
\[
P_+(B+S)-P_+(B-S),
\]
lies in $C^*(X)^G.$ Thus the formal difference $[P_+(B+S)]-[P_+(B-S)]$ determines a class in  $K_0(C^*(X)^G)$.
The signature of  $(H_{*,X}, b, S)^G$  is then defined to be the class in  $K_0(C^*(X)^G)$ determined by
			\[
			[P_+(B+S)]-[P_+(B-S)].
			\]
		\end{enumerate}
	\end{definition}

In the following, we denote the signature of  $(H_{*,X}, b, T)^G$, an $n$-dimensional $G$-equivariant analytically controlled Hilbert-Poincar\'e complex over $X$, by
\[
\text{Ind} (H_{*,X}, b, T)^G\in K_n(C^*(X)^G).
\]

\subsection{Homotopy invariance of the signature of Hilbert-Poincar\'e complexes}\label{subsec homo inv of sig of complex}

In this subsection, we recall the proof of the homotopy invariance of the signature of $G$-equivariant analytically controlled Hilbert-Poincar\'e complexes.

%\begin{definition}[\cite{HR051}, Definition 4.4]
%		Let $(H_{*,X}, b)^G$ be a $G$-equivariant Hilbert complex over $X$. An operator homotopy of $G$-equivariant  analytically controlled  Hilbert-Poincar\'e complex over $X$  is a norm continuous family of adjointable operators $T_s, \ (s\in [0,1])$ such that each $(H_{*,X}, b, T_s)^G$ is a   $G$-equivariant  analytically controlled  Hilbert-Poincar\'e complex .
%	\end{definition}
%	
%	\begin{lemma}[\cite{HR051},Lemma 4.6 ]\label{lem operator homotopy}
%		If a duality operator $T$ on a Hilbert-Poincar\'e complex is operator homotopic to $-T$, then the signature of $(E, b, T)$ is trivial.
%	\end{lemma}

	Let
\[
f: (H'_{*,X}, b', T')^G\to (H''_{*,X}, b'', T'')^G
\]
be a $G$-equivariant analytically controlled homotopy equivalence between two $G$-equivariant analytically controlled Hilbert-Poincar\'e complexes over $X.$ Recall that the chirality duality operator $S'=\gamma T'$ and $S''=\gamma T''.$
 Then
	\begin{equation}\label{eq a complex in 4}
	(H'_{*,X}\oplus H''_{*,X}, \left(\begin{array}{cc}
	b' & 0\\
	0 & b''
	\end{array} \right), \left(\begin{array}{cc}
	T' & 0\\
	0 & -T''
	\end{array} \right) )^G
	\end{equation}
	is a $G$-equivariant analytically controlled Hilbert-Poincar\'e complex over $X$.
	Higson and Roe built an explicit homotopy path connecting the representative of
\[
\text{Ind}(H'_{*,X}\oplus H''_{*,X}, \left(\begin{array}{cc}
	b' & 0\\
	0 & b''
	\end{array} \right), \left(\begin{array}{cc}
	T' & 0\\
	0 & -T''
	\end{array} \right) )^G
\]
to the identity or zero element in \cite{HR051}. We describe this homotopy path in details for the odd dimensional case only. The even dimensional case is completely similar.
	Set
	\[
	B=\left(\begin{array}{cc}
	b' & 0\\
	0 & b''
	\end{array} \right)+\left(\begin{array}{cc}
	b' & 0\\
	0 &b''
	\end{array} \right)^*, S= \left(\begin{array}{cc}
	S' & 0\\
	0 & -S''
	\end{array} \right) .
	\]
	Then the  signature of complexes defined in line \eqref{eq a complex in 4} is represented by
	\[
	\frac{B+S}{B-S}.
	\]
	
	From \cite{HR051} and \cite{WXY18}, we know that the following are all $G$-equivariant analytically controlled Hilbert-Poincar\'e complexes over $X$:
	\[(H'_{*,X}\oplus H''_{*,X}, \left(\begin{array}{cc}
	b' & 0\\
	0 & b''
	\end{array} \right), T_f(s))^G, s\in[0,\frac{2}{3}],\]
where $T_f(s)$ equals
\[
\left(\begin{array}{cc}
	T' & 0\\
	0 & (3s-1)T'' -3sfT'f^*
	\end{array} \right)
\]
for $s\in [0,\frac{1}{3}],$ and equals
\[
\left(\begin{array}{cc}
	\cos((3s-1)\frac{\pi}{2})T' & \sin((3s-1)\frac{\pi}{2})T'f^* \\
	\sin((3s-1)\frac{\pi}{2})gT' & -\cos((3s-1)\frac{\pi}{2})fT'f^*
	\end{array} \right)
\]
for $s\in  [\frac{1}{3},\frac{2}{3}].$
	Thus the following
	\[
			\frac{B+ S_f(s) }{B- S_f(s)},  s\in [0,\frac{2}{3}]
      \]
	forms an invertible path in  $C^*(X)^G,$ where $S_f(s)$ is the corresponding chirality duality operator of $T_f(s).$

	Note that the following are still $G$-equivariant analytically controlled Hilbert-Poincar\'e complexes over $X$:
	\[(H'_{*,X}\oplus H''_{*,X}, \left(\begin{array}{cc}
	b' & 0\\
	0 & b''
	\end{array} \right), \left(\begin{array}{cc}
	0 & e^{i s} T'f^* \\
	e^{-i s} fT_{\widetilde{M}} & 0
	\end{array} \right) )^G, \ s\in [0,1].
	\]
	Thus we can connect
\[
\frac{B+ S_f(\frac{2}{3}) }{B- S_f(\frac{2}{3})}
\]
to the identity by the path
	\[
	\frac{B+\left(\begin{array}{cc}
		0 & S_{\widetilde{M}}f^* \\
		fS_{\widetilde{M}} & 0
		\end{array} \right)}{B-\left(\begin{array}{cc}
		0 & e^{i(3s-2)\pi } S_{\widetilde{M}}f^* \\
		e^{-i(3s-2)\pi }fS_{\widetilde{M}} & 0
		\end{array} \right)},\ \  s\in [\frac{2}{3},1].
	\]
	In a word,, we obtain an invertible path in $C^*(X)^G$ connecting
\[
\frac{B+S}{B-S}
\]
 to the identity. In the following, we will denote this path by
	\begin{equation}\label{eq odd sig homo path}
	\frac{B_f+ S_f}{B_f-S_f}(s),\ \  s\in [0,1]
     \end{equation}
	where
	\[
	\frac{B_f+ S_f}{B_f-S_f}(0)= \frac{B+S}{B-S},\ \ \frac{B_f+ S_f}{B_f-S_f}(1)= I.
	\]
Note that this path is derived from a continuous family of $G$-equivariant analytically controlled Hilbert-Poincar\'e complexes, which will be denoted as
\begin{equation}\label{eq sig homo complex path}
(H'_{*,X}\oplus H''_{*,X} , \begin{pmatrix}
b' & 0 \\
0 & b''
\end{pmatrix}, T_f(s))^G, s\in [0,1].
\end{equation}

In even case, the path will be denoted by
	\begin{equation}\label{eq even sig homo path}
	P_+(B_f+S_f)-P_+(B_f-S_f).
	\end{equation}

The path defined above actually proves the homotopy invariance of the signature of Hilbert-Poincar\'e complexes, i.e.
\begin{proposition}[Theorem 5.12, \cite{HR051}]\label{prop homo inv of sig of com}
Let
\[
f: (H'_{*,X}, b', T')^G\to (H''_{*,X}, b'', T'')^G
\]
be a $G$-equivariant analytically controlled homotopy equivalence between two $n$-dimensional $G$-equivariant analytically controlled Hilbert-Poincar\'e complexes over $X,$ then we have
\[
\text{Ind}(H'_{*,X}, b', T')^G=\text{Ind}(H''_{*,X}, b'', T'')^G\in K_n (C^*(X)^G).
\]
\end{proposition}
\begin{proof}
We prove this proposition for the odd case only, the even case is parallel. Set $B'=b'+(b')^*$ and $B''=b''+(b'')^*$. Then it is sufficient to consider the path
\[
\begin{pmatrix}
\frac{B'+S'}{B'-S'}  &0\\
0  & I
\end{pmatrix}\left(\frac{B_f+ S_f}{B_f-S_f}\right)^{-1}(1-s), s\in [0,1].
\]
\end{proof}

\subsection{Analytically controlled Hilbert-Poincar\'e pair}

In this subsection, we recall the definition of the $G$-equivariant analytically controlled Hilbert-Poincar\'e pair, which is used in the next subsection to prove the bordism invariance of the  signature of complexes, and in the next section to define the relative signature.

Let $X$ be a proper metric space and $G$ be a discrete group acting on  $X$ freely, cocompactly, and properly.
\begin{definition}[Definition 7.2, \cite{HR051}]\label{def hil poin pair}
		An $(n+1)$-dimensional $G$-equivariant analytically controlled Hilbert-Poincar\'e pair over $X$ is a  $G$-equivariant analytically controlled Hilbert complex $(H_{*,X}, b)^G,$ together with a $G$-equivariant analytically controlled operator $T: H_{*,X}\to H_{n+1-*.X}$ and a $G$-equivariant analytically controlled projection $P:  H_{*,X}\to  H_{*,X}$ such that
		\begin{enumerate}
			\item  $PbP=bP$, hence the orthogonal projection $P$ determines a subcomplex, $(PH_{*,X}, Pb)^G$, of $(H_{*,X}, b)^G$.  Note that $ bP^\perp=P^\perp bP^\perp,$ thus the complex $(P^\perp H_{*,X}, P^{\perp}b)^G$  is  the corresponding quotient complex of the subcomplex $(P H_{*,X}, Pb)^G.$
			\item The range of the operator $Tb^*+(-1)^pbT: H_{p,X}\to H_{n-p,X}$ is contained within the range of $P: H_{n-p,X}\to H_{n-p,X}$.
			\item $T^* = (-1)^{p(n+1-p)}T: H_{p,X}\to H_{n+1-p,X}$.
			\item $P^\perp T$ is a $G$-equivariant analytically controlled chain homotopy equivalence from the dual complex $(H_{*,X}, b^*)^G$ to $ (P^\perp H_{*,X}, P^\perp b)^G$.
		\end{enumerate}
We will denote this pair by
\[
(H_{*,X}, b, T, P)^G.
\]
	\end{definition}

Note that by definition,
\[
Pb=b: PH_{p,X}\to  PH_{p-1,X},
\]
hence $(PH_{*,X},Pb)^G$ is a $G$-equivariant analytically controlled Hilbert complex over $X$.
Correspondingly, the adjoint of $Pb$ is
\[
Pb^*:  PH_{p-1,X}\to  PH_{p,X},
\]
and the dual complex of $(PH_{*,X},Pb)^G$ is $(PH_{n-*,X},Pb^*)^G.$

The next lemma plays a central role in formulating the bordism invariance of the signature of complexes.
	\begin{lemma}[Lemma 7.4, \cite{HR051}]
		Let $(H_{*,X}, b, T, P)^G$ be an $n+1$ dimensional $G$-equivariant analytically controlled Hilbert-Poincar\'e  pair. Then the operator $T_0=Tb^*+(-1)^pbT: H_{p,X}\to H_{n-p,X}$ satisfies the following conditions:
		\begin{enumerate}
			\item $T_0^*=(-1)^{(n-p)p}T_0: H_{p,X}\to H_{n-p,X}$.
			\item $T_0=PT_0=T_0P$.
			\item $T_0b^*(v)+(-1)^pbT_0(v)=0, \forall v\in  PH_{p,X}$.
			\item $T_0$ induces a $G$-equivariant analytically controlled homotopy equivalence from  $(PH_{n-*,X}, Pb^*)^G$ to $(PH_{*,X}, Pb)^G.$
		\end{enumerate}
	\end{lemma}

	The above lemma asserts that $(PH_{*,X}, Pb, T_0 )^G$ is a $G$-equivariant analytically controlled Hilbert-Poincar\'e complex, which  will be called the boundary complex of the pair $(PH_{*,X}, b, T, P)^G$.
\subsection{Bordism invariance of the signature of Hilbert-Poincar\'e complexes}\label{subsec bord inv of sig of complex}

In this subsection, we recall the  formulation and the proof of the bordism invariance of the signature of  $G$-equivariant analytically controlled Hilbert-Poincar\'e complexes.
	
The following proposition formulates the bordism invariance of the signature of $G$-equivariant analytically controlled Hilbert-Poincar\'e  complexes.
\begin{proposition}[Theorem 7.6, \cite{HR051}]\label{prop bord inv of sig of com}
Let  $(H_{*,X}, b, T, P)^G$ be an $n+1$ dimensional $G$-equivariant analytically controlled Hilbert-Poincar\'e  pair over $X$, $(PH_{*,X}, Pb, T_0 )^G$ be its boundary complex. Then we have
\[
\text{Ind}(PH_{*,X}, Pb, T_0 )^G=0\in K_n(C^*(X)^G).
\]
\end{proposition}

We briefly recall the proof of the above Proposition as follows. Set
	\[
	\overline{H}_{p,X}= H_{p,X} \oplus P^\perp H_{p+1,X},\ \  \overline{b}(\lambda) =
	\left(
	\begin{array}{cc}
	b & 0\\
	\lambda P^\perp & P^\perp b
	\end{array}
	\right), \ \lambda\in [-1,0].
	\]
	Then $(\overline{H}_{p,X},\overline{b}(\lambda))^G$ is a $G$-equivariant analytically controlled Hilbert-Poincar\'e complex over $X$. The following family of operators
	\[
	\overline{T}(s)=\left(
	\begin{array}{cc}
	0 & e^{is\pi}TP^\perp \\
	(-1)^pe^{-is\pi} P^\perp T & 0
	\end{array}
	\right): \overline{H}_{p,X}\to \overline{H}_{n-p,X}
	\]
	are $G$-equivariant analytically controlled duality operators of  $(\overline{H}_{*,X}, \overline{b}(\lambda))^G$ as long as $\lambda s=0$, i.e.
\begin{equation}\label{eq rel homology complex}
(\overline{H}_{*,X}, \overline{b}(\lambda), \overline{T}(s))^G
\end{equation}
 is a $G$-equivariant analytically controlled Hilbert-Poincar\'e complex as long as $\lambda s=0$.

	 Note that
\begin{eqnarray*}
A: PH_{*,X} &\to& \overline{H}_{*,X}=H_{*,X}\oplus P^\perp H_{*+1,X}\\
A(v)&=&v\oplus 0
\end{eqnarray*}
 defines a $G$-equivariant analytically controlled chain homotopy equivalence
	\[
	A: (PH_{*,x}, Pb, T_0)^G\to (\overline{H}_{*,X}, \overline{b}(-1), \overline{T}(0))^G.
	\]
	Moreover, for $(\overline{H}_{*,X}, \overline{b}(0))^G$, Poincar\'e duality operator $\overline{T}(0)$ is connected to $\bar{T}(1)=-\bar{T}(0)$ along the path of Poincar\'e duality operators $\bar{T}(s), s\in [0,1].$

	Thus, we obtain a path connecting the representative of the signature of $(PH_{*,X}, Pb, T_0)$ to the trivial element. When $n$ is odd,  we denote this path by
	\begin{equation}\label{eq odd sig bord path}
\frac{B_{P} + S_P }{ B_{P} - S_P  },
      \end{equation}
	where
	$$\frac{B_{P} + S_P }{ B_{P} - S_P  }(t),\  t\in [0,1]$$
	equals
\[
\begin{pmatrix}
\frac{Pb+Pb^*+S_0}{Pb+Pb^*-S_0} & 0\\
0 & I
\end{pmatrix}
\left(
\frac{B_A +S_A}{B_A-S_A}
\right)^{-1}(1-3t)
\]
when $t\in[0,\frac{1}{3}],$ equals
%	\[\begin{pmatrix}
%I & 0\\
%0 &	\frac{\begin{pmatrix}
%		b+b^* & -(2-3t) P^\perp\\
%		-(2-3t) P^\perp & P^\perp b +  P^\perp b^*
%		\end{pmatrix} +\begin{pmatrix}
%		0 & SP^\perp \\
%		(-1)^p P^\perp S & 0
%		\end{pmatrix} }{\begin{pmatrix}
%		b+b^* & -(2-3t) P^\perp\\
%		-(2-3t) P^\perp & P^\perp b +  P^\perp b^*
%		\end{pmatrix} -
%		\begin{pmatrix}
%		0 & SP^\perp \\
%		(-1)^p P^\perp S & 0
%		\end{pmatrix} }
%\end{pmatrix}\]
\[\begin{pmatrix}
I & 0\\
0 &	\frac{\bar{b}(3t-2)+\bar{b}^*(3t-2)+\bar{S}(0) }{\bar{b}(3t-2)+\bar{b}^*(3t-2)-\bar{S}(0) }
\end{pmatrix}\]
	when $ t\in [\frac{1}{3}, \frac{2}{3}] $, and equals
%\[
%\begin{pmatrix}
%I & 0 \\
%0	& \frac{
%		\begin{pmatrix}
%		b+b^* & 0 \\
%		0 & P^\perp b +  P^\perp b^*
%		\end{pmatrix} +
%		\begin{pmatrix}
%		0 & SP^\perp \\
%		(-1)^p P^\perp S & 0
%		\end{pmatrix}
%		}{
%		\begin{pmatrix}
%		b+b^* &  0\\
%		0 & P^\perp b +  P^\perp b^*
%		\end{pmatrix}
%		 -
%		\begin{pmatrix}
%		0 & e^{i (3t-2)\pi }SP^\perp \\
%		(-1)^p e^{-i (3t-2)\pi }P^\perp S & 0
%		\end{pmatrix} }
%\end{pmatrix}\]
\[
\begin{pmatrix}
I & 0 \\
0	& \frac{\bar{b}(0)+\bar{b}^*(0)+\bar{S}(0) }{\bar{b}(0)+\bar{b}^*(0)-\bar{S}(3t-2) }
\end{pmatrix}\]
	when $t\in [\frac{2}{3}, 1]$.
	
	Similarly, in even case, the path will be denoted by
	\begin{equation}\label{eq even sig bord path}
	P_+(B_P+ S_P)-P_+(B_P-S_P).
	\end{equation}
Note that the above path proving the bordism invariance of the signature of Hilbert-Poincar\'e complexes is generated from a continuous family of Hilbert-Poincar\'e complex, which will be denoted as
\begin{equation}\label{eq sig bord complex path}
(\overline{H}_{p,X}, \overline{b}(\lambda), \overline{T}(s))^G, \lambda\in [-1,0], s\in [0,1], \lambda s=0.
\end{equation}

\subsection{Signature of compact PL manifolds}\label{subsec sig of mani}

In this subsection, we introduce the definition of the signature of compact PL manifolds.

For an $n$-dimensional compact PL manifold $N$ with fundamental group $G$, let $\widetilde{N}$ be the universal convering space of $N$. Equip $\widetilde{N}$ with a $G$-invariant triangulation $\text{Tri}(\widetilde{N} )^G$.  The  $L^2$-completion of the  simplicial chain complex $(E_*(\widetilde{N}) , b_{\widetilde N})$ given by the triangulation then induces a $G$-equivariant analytically controlled Hilbert complex over $\widetilde{N}$,
\[(L^2(E_*(\widetilde{N})), b_{\widetilde N})^G.\]
Equipped with the Poincar\'e duality map $T_{\widetilde N}$ which is  given by  the usual cap product with the fundamental class $[\widetilde N]$,
\[
(L^2(E_*(\widetilde{N})), b_{\widetilde N},T_{\widetilde N} )^G
\]
defines a $G$-equivariant analytically controlled Hilbert-Poincar\'e complex over $\widetilde{N}.$

\begin{definition}
Let $N$ be an $n$-dimensional compact PL manifold with fundamental group $G$, and $\widetilde{N}$ be the universal covering space of $N.$ Take a  $G$-invariant triangulation $\text{Tri}(\widetilde{N} )$ of $\widetilde{N}.$ Consider
\[
(L^2(E_*(\widetilde{N})), b_{\widetilde N},T_{\widetilde N} )^G,
\]
 the corresponding $G$-equivariant analytically controlled Hilbert-Poincar\'e complex over $\widetilde{N}.$ Then we define $\text{Ind}(N)\in K_n(C^*(\widetilde N)^G),$ the signature of $N,$ to be the signature of the complex
\[
\text{Ind}(L^2(E_*(\widetilde{N})), b_{\widetilde N},T_{\widetilde N} )^G.
\]
It is well defined since the signature of  $G$-equivariant analytically controlled Hilbert-Poincar\'e complexes is homotopy invariant.
\end{definition}	

By the argument in the Subsection \ref{subsec homo inv of sig of complex}, we know that the signature  is a homotopy invariant of compact  PL manifolds.

On the other hand, the argument in the Subsection \ref{subsec bord inv of sig of complex} proves that the signature of compact PL manifolds is a bordism invariant. In fact, let $(N, \partial N)$ be an $n+1$-dimensional compact PL manifold with boundary, let $\Gamma$ be the fundamental group of $N$ and $G$ be the fundamental group pf $\partial N.$ Let $p:\widetilde N\to N$  be the universal covering of $N$, and $\widetilde{\partial N}$ be the universal covering space of $\partial N.$ Let $\widetilde{\partial N}'=p^{-1}\partial N$ be the $\Gamma$-Galois covering space of $\partial N$. Take a triangulation $\text{Tri}(N, \partial N)$ of  $(N,\partial N).$ Then one can lift $\text{Tri}(N, \partial N)$ up to a$(\widetilde N, \widetilde{\partial N}')$ as a $\Gamma$-equivariant triangulation $\text{Tri}(\widetilde N, \widetilde{\partial N}')^{\Gamma}$, and lift the restriction of  $\text{Tri}(N, \partial N)$ on $\partial N$ up to $\widetilde{\partial N}$ as a $G$-equivariant triangulation $\text{Tri}( \widetilde{\partial N})^G.$ Then the $L^2$-completion of the simplicial chain complex $(L^2(E_*(\widetilde{N})) , b_{\widetilde N})^\Gamma$ induced by $\text{Tri}(\widetilde N, \widetilde{\partial N}')^{\Gamma}$ forms a $\Gamma$-equivariant analytically controlled Hilbert complex over $\widetilde N$. Consider the Poincar\'e duality operator $T$ induces by the cap product with the fundamental class $[\widetilde N]$ and the usual projection $P$ onto the complex on $\widetilde{\partial N}'$, the following
\[
(L^2(E_*(\widetilde{N})) , b_{\widetilde N}, T, P)^\Gamma
\]
becomes a $\Gamma$-equivariant analytically controlled Hilbert-Poincar\'e pair over $\widetilde N$. Parallelly, we have the following $G$-equivariant analytically controlled Hilbert-Poincar\'e complex over $\widetilde{\partial N}$,
\[(L^2(E_*(\widetilde{\partial N})) , b_{\widetilde{ \partial {N}}}, T_{\partial})^G\]
which is consists of the $L^2$-completion of the simplicial chain complex of $\text{Tri}( \widetilde{\partial N})^G,$ and the Poincar\'e duality operator induced by the cap product with $[\widetilde {\partial N}].$  Then under the homeomorphism $\iota,$ defined in line \eqref{eq iota of manifold}, Subsection \ref{subsec relative C alg}, we have
\[
\iota (L^2(E_*(\widetilde{\partial N})) , b_{\widetilde{ \partial {N}}}, T_{\partial})^G=(PL^2(E_*(\widetilde{N})) , Pb_{\widetilde N}, T_0)^\Gamma.
\]
Thus under the $K$-theory map $\iota^*,$ which is induced by the $C^*$-map
\[\iota: C^*(\widetilde{\partial N})^G\to C^*(\widetilde{N})^{\Gamma},\]
we have
\[
\iota^* \text{Ind}(L^2(E_*(\widetilde{\partial N})) , b_{\widetilde{ \partial {N}}}, T_{\partial})^G=\text{Ind}(PL^2(E_*(\widetilde{N})) , Pb_{\widetilde N}, T_0)^\Gamma\in K_{n}(C^*(\widetilde N)^{\Gamma}).
\]
The right hand side is shown to be trivial in Proposition \ref{prop bord inv of sig of com}.

\section{Relative signature  and mapping relative $L$-theory to $K$-theory}

In this section, we define the relative signature of compact PL manifolds with boundary. We will also prove its homotopy invariance and bordism invariance. At last, by the relative signature, we define the group homomorphism from the relative $L$-theory to the $K$-theory.

In this section, we consider even dimensional compact PL manifolds with boundary only, the odd dimensional case is completely parallel.

\subsection{Relative signature of compact PL manifolds with boundary and its homotopy invariance}

In this subsection, we define the relative signature of compact PL manifolds with boundary, and prove its homotopy invariance.

Let $(N, \partial N)$ be an $n=2k$-dimensional compact PL manifold with boundary, let $\Gamma$ be the fundamental group of $N$ and $G$ be the fundamental group of $\partial N.$ Let $p:\widetilde N\to N$  be the universal covering of $N$, and $\widetilde{\partial N}$ be the universal covering space of $\partial N.$ Let $\widetilde{\partial N}'=p^{-1}\partial N$ be the $\Gamma$-Galois covering space of $\partial N$. Take a triangulation $\text{Tri}(N, \partial N)$ of  $(N,\partial N).$ As the construction in the end of Subsection \ref{subsec sig of mani}, one can lift $\text{Tri}(N, \partial N)$ up to $(\widetilde N, \widetilde{\partial N}')$ as a $\Gamma$-equivariant triangulation $\text{Tri}(\widetilde N, \widetilde{\partial N}')^{\Gamma}$, lift the restriction of  $\text{Tri}(N, \partial N)$ on $\partial N$ up to $\widetilde{\partial N}$ as a $G$-equivariant triangulation $\text{Tri}( \widetilde{\partial N})^G.$ Then we obtain a $\Gamma$-equivariant analytically controlled Hilbert-Poincar\'e pair over $\widetilde N,$
\[
(L^2(E_*(\widetilde{N})) , b_{\widetilde N}, T, P)^\Gamma
\]
and a  $G$-equivariant analytically controlled Hilbert-Poincar\'e complex over $\widetilde {\partial N}$,
\[(L^2(E_*(\widetilde{\partial N})) , b_{\widetilde{ \partial {N}}}, T_{\partial})^G\]
 such that
\[
\iota (L^2(E_*(\widetilde{\partial N})) , b_{\widetilde{ \partial {N}}}, T_{\partial})^G=(PL^2(E_*(\widetilde{N})) , Pb_{\widetilde N}, T_0,)^\Gamma.
\]

 Let
\[\frac{B+S}{B-S}\]
be the representative of the signature of
\[(L^2(E_*(\widetilde{\partial N})) , b_{\widetilde{ \partial {N}}}, T_{\partial})^G\]
defined in Theorem \ref{def signature of hp complex},
and
\[
\frac{B_P+S_P}{B_P-S_P}
\]
be the path  defined in line \eqref{eq odd sig bord path} , then
\begin{equation}\label{eq 2k rel sig class}
(\begin{pmatrix}
\frac{B+S}{B-S} & 0 \\
0 & I
\end{pmatrix}, \frac{B_P+S_P}{B_P-S_P}
)
\end{equation}
defines an invertible element in $C^*_\iota.$ Recall that $[v]$ is the generator class of $K_1(C(S^1)),$ then
\begin{equation}
[(\begin{pmatrix}
\frac{B+S}{B-S} & 0 \\
0 & I
\end{pmatrix}, \frac{B_P+S_P}{B_P-S_P}
)]\otimes [v]
\end{equation}
defines a class in $K_{n}(C^*(N, \partial N)^{\Gamma, G}).$

\begin{theorem}\label{theo define relative sig class}
The class
\[
[(\begin{pmatrix}
\frac{B+S}{B-S} & 0 \\
0 & I
\end{pmatrix}, \frac{B_P+S_P}{B_P-S_P}
)]\otimes [v]\in K_{0}(C^*(N, \partial N)^{\Gamma, G})
\]
 we defined above is independent of the choice of the triangulation. We call this class the relative signature of $(N,\partial N),$ and denote it by
\[
\text{relInd}(N, \partial N).
\]
\end{theorem}

\begin{proof}
Let $\text{Tri}'(N, \partial N)$ and $\text{Tri}''(N,\partial N)$ be two triangulations, then their corresponding $\Gamma$-equivariant analytically controlled Hilbert-Poincar\'e pair over $\widetilde N$ are
\[
(L^2(E_*(\widetilde{N})') , b'_{\widetilde N}, T', P')^\Gamma
\]
and
\[
(L^2(E_*(\widetilde{N})'') , b''_{\widetilde N}, T'', P'')^\Gamma
\]
respectively, and their corresponding $G$-equivariant analytically controlled Hilbert-Poincar\'e complex over $\widetilde{\partial N}$ are
\[(L^2(E_*(\widetilde{\partial N}))' , b'_{\widetilde{ \partial {N}}}, T'_{\partial})^G\]
and
\[(L^2(E_*(\widetilde{\partial N}))'' , b''_{\widetilde{ \partial {N}}}, T''_{\partial})^G\]
respectively.

Let $f:(E_*(\widetilde{N})', b')\to (E_*(\widetilde{N})'', b'')$ be the homotopy equivalence between these two simplicial chain complexes, note that
\[
P'' fP' = fP'.
\]
Thus $f$ induces  the  analytically controlled homotopy equivalence
\[
f: (L^2(E_*(\widetilde{\partial N}))', b'_{\widetilde{ \partial {N}}}, T'_{\partial})^G\to (L^2(E_*(\widetilde{\partial N}))'' , b''_{\widetilde{ \partial {N}}}, T''_{\partial})^G.
\]
The following is also  an analytically controlled homotopy equivalence induced by $f$,
\[
f: (\overline{L^2(E_*(\widetilde{N})')},\overline{b'}_{\widetilde N}(\lambda),\overline{T'}(s))^{\Gamma}\to (\overline{L^2(E_*(\widetilde{N})'')},\overline{b''}_{\widetilde N}(\lambda),\overline{T''}(s))^\Gamma,
\]
where the above complexes are defined in line \eqref{eq sig bord complex path}, and
\[\lambda\in [-1, 0], s\in [0,1].\]
Then the theorem follows from a verbatim application of the construction in Subsection \ref{subsec homo inv of sig of complex}.
\end{proof}

By the same reason, we have
\begin{theorem}\label{thm homo inv of rel sig}
The signature of $n=2k$-dimensional compact PL manifolds  with boundary defined in Theorem \ref{theo define relative sig class}  is a homotopy invariant. That is, let
\[
f: (M,\partial M)\to (N, \partial N)
\]
be a homotopy equivalence of compact PL manifolds with boundary, and $\Gamma$ be the fundamental group of $N,$ $G$ be the fundamental group of $\partial N,$ then
\[
\text{relInd}(M,\partial M)=\text{relInd}(N,\partial N)\in K_{0}(C^*(\widetilde N, \widetilde{\partial N})^{\Gamma, G}).
\]
\end{theorem}

\subsection{Controlled Hilbert-Poincar\'e triple}\label{subsec hil poin triple}

In this subsection, we introduce the notion of the $G$-equivariant analytically controlled Hilbert-Poincar\'e triple, which will be used  to  formulate and prove the bordism invariance of the relative signature of compact  PL manifolds with boundary.

In this subsection, let $X$ be a proper metric space and $G$ be a discrete group acting on $X$ freely, cocompactly, and properly.
	
	\begin{definition}\label{def Hilbert Poincare 2-ads}
		 An $(n+2)$-dimensional $G$-equivariant analytically controlled Hilbert-Poincar\'e triple over  $X$ consists of an $n+2$-dimensional $G$-equivariant analytically controlled Hilbert complex $(H_{*,X}, b)^G$ over $X$, a $G$-equivariant analytically controlled maps $T: H_{*,X}\to H_{n+2-*,X}$ and $G$-equivariant analytically  controlled projections $P_\pm H_{*, X}\to H_{*,X}$ such that
		\begin{enumerate}
			\item $P_\pm bP_\pm =bP_\pm$.
			\item $P_\vee=P_+\vee P_-$, and  $(H_{*,X}, b, T, P_\vee)^G$ is an $(n+2)$-dimensional $G$-equivariant analytically controlled Hilbert-Poincar\'e pair. Set $(P_\vee H_{*,X}, P_\vee b, T_0)^G$ as its boundary complex.
			\item $P_\wedge=P_+\wedge P_-$, and
			$ (P_\pm  H_{*,X}, P_\pm P b, P_\pm T_0 P_\pm, P_\wedge)^G $ are $(n+1)$-dimensional $G$-equivariant analytically controlled Hilbert-Poincar\'e pairs, and their boundary complexes  are $G$-equivariant analytically controlled homotopy equivalence to each other.
			\item $P^\perp_\mp T P^\perp_\pm : (P_\pm^\perp H_{*,X}, P_\pm^\perp b)\to (P_\mp^\perp H_{*,X}, P_\mp^\perp b)$ are  $G$-equivariant analytically controlled homotopy equivalence of complexes.
		\end{enumerate}
In the following, we shall denote a  $G$-equivariant analytically controlled Hilbert-Poincar\'e triple over $X$ consists of elements defined above as
\[
(H_{*,X}, b, T, P_\pm)^G.
\]
		
	\end{definition}
	
	\begin{remark}\label{remark for dual of b and projection}
		Note that in general, $P_\pm b^*P_\pm \neq b^*P_\pm $, however, there is $P^\perp_\pm b^*P^\perp_\pm =b^*P^\perp_\pm$. In fact, decompose $H_{*,X}$ as $P_\pm H_{*,X}\oplus P^\perp_\pm H_{*,X},$ then $P_\pm bP_\pm =bP_\pm$ implies that
		\[
		b=\left(
		\begin{array}{cc}
		b_{11} & b_{12}\\
		0   & b_{22}
		\end{array}
		\right),
		\]
		thus we have
		\[
		b^*= \left(
		\begin{array}{cc}
		b^*_{11} & 0\\
		b^*_{12}   & b^*_{22}
		\end{array}
		\right).
		\]
	\end{remark}

	\begin{lemma}\label{lem 4 complex positive}	
Let
\[
(H_{*,X}, b, T, P_\pm)^G
\]
be an $n+2$-dimensional $G$-equivariant analytically controlled Hilbert-Poincar\'e triple over $X$.
Set
		\begin{eqnarray*}
			 \widehat{H}_{+,*,X} &=& H_{*, X}\oplus P_+^\perp H_{*+1,X} \oplus P_-^\perp H_{*+1, X}\oplus P_\vee^\perp H_{*+2,X},  *=0,1,\cdots n, \\
			\widehat{b}_+(\lambda,\mu)&=&\left(
			\begin{array}{cccc}
				b &0 & 0&0 \\
				\mu P_+^\perp   &   -P_+^\perp  b& 0&0 \\
				\lambda P_-^\perp  & 0 & -P_-^\perp b &0 \\
				0	& \lambda P_\vee^\perp & -\mu P_\vee^\perp &  P_\vee^\perp b
			\end{array}
			\right)
\end{eqnarray*}
and \[ \widehat{T}_{+}(s)  =   \left(
			\begin{array}{cccc}
				0&0 & 0&  e^{-i\pi s} T P_\vee^\perp   \\
				0& 0&  (-1)^{p-1}e^{-i\pi s}P_+^\perp T  P_-^\perp  & 0\\
				0&  (-1)^{p}e^{i\pi s}P_-^\perp T P_+^\perp &0 & 0\\
				e^{i\pi s}P_\vee^\perp T& 0& 0& 0
			\end{array}
			\right)
		\]
on $ \widehat{H}_{+,p,X}.$
Then $(\widehat{H}_{+,*,X}, \widehat{b}_+(\lambda, \mu),  \widehat{T}_{+}(s))^G$
 defines an $n$-dimensional $G$-equivariant analytically controlled Hilbert-Poincar\'e complex over $X$ as long as
		\begin{enumerate}
			\item $ \lambda,\mu\in[-1,0], s \in [0,1]$.
			\item $\lambda s=0.$
%\item $\lambda=0$ if $\mu =0$.
		\end{enumerate}
	\end{lemma}	
	
	\begin{proof}
		By direct computation, one can see that $(\widehat{H}_{+,*,X}, \hb_+(\lambda, \mu))^G$ is a $G$-equivatiant analytically controlled Hilbert complex over $X$. Thus it is sufficient to show that $\hT_+(s)$ are controlled Hilbert-Poincar\'e dualities when it is satisfied that
	\begin{enumerate}
			\item $ \lambda,\mu\in[-1,0], s \in [0,1]$.
			\item $\lambda s=0$.
%\item $\lambda=0$ if $\mu =0$.
		\end{enumerate}
We focus on $s=0$ case first.

We claim that $\hT_+^*(0)= (-1)^{(n-p)p}\hT_+(0)$.In fact, for $\widehat{T}_{+}(0)$ on $\widehat{H}_{+,p,X},$ we have
\begin{eqnarray*}
\widehat{T}_{+}^*(0)&=& \left(
			\begin{array}{cccc}
				0&0 & 0& T P_\vee^\perp   \\
				0& 0&  (-1)^{p-1}P_+^\perp T  P_-^\perp  & 0\\
				0&  (-1)^{p}P_-^\perp T P_+^\perp &0 & 0\\
				P_\vee^\perp T& 0& 0& 0
			\end{array}
			\right)^*\\
&=&  \left(\begin{array}{cccc}
				0&0 & 0& T^* P_\vee^\perp   \\
				0& 0&  (-1)^{p}P_+^\perp T^*  P_-^\perp  & 0\\
				0&  (-1)^{p-1}P_-^\perp T^* P_+^\perp &0 & 0\\
				P_\vee^\perp T^*& 0& 0& 0
			\end{array}
			\right).
\end{eqnarray*}
Now the claim follows from
\begin{eqnarray*}
 T^* P_\vee^\perp&=&   (-1)^{(n-p)p}T P_\vee^\perp: H_{n+2-p, X}\mapsto  H_{p, X} ,\\
P_+^\perp T^*  P_-^\perp &=& (-1)^{(n+1-p)(p+1)}P_+^\perp T P_-^\perp:  H_{n+1-p, X} \mapsto  H_{p+1, X},\\
P_-^\perp T^* P_+^\perp &=& (-1)^{(n+1-p)(p+1)}P_-^\perp T P_+^\perp:  H_{n+1-p, X} \mapsto  H_{p+1, X},\\
P_\vee^\perp T^* &=&  (-1)^{(n-p)p} P_\vee^\perp T:  H_{n-p, X} \mapsto  H_{p+2, X}
\end{eqnarray*}

Then we need to show that
$$(-1)^p\hb_+(\lambda,\mu)\hT_+(0) +\hT_+(0)\hb^*_+(\lambda,\mu)=0.$$ Set
\[
V_p=v_1\oplus v_2\oplus v_3\oplus v_4\in \widehat{H}_{+,p,X},
\]
we have
\begin{eqnarray*}
&&\hb_+(\lambda,\mu)\hT_+(0) V_p\\
&=&\hb_+(\lambda,\mu) \begin{pmatrix}
				0&0 & 0&   T P_\vee^\perp   \\
				0& 0&  (-1)^{p-1}P_+^\perp T  P_-^\perp  & 0\\
				0&  (-1)^{p}P_-^\perp T P_+^\perp &0 & 0\\
				P_\vee^\perp T& 0& 0& 0
			\end{pmatrix}
\begin{pmatrix}
v_1\\
v_2\\
v_3\\
v_4
\end{pmatrix}\\
&=& \begin{pmatrix}
				0&0                   &  0                                & b T P_\vee^\perp   \\
				0& 0                  &  (-1)^{p}P_+^\perp bT  P_-^\perp  & \mu P_+^\perp TP_\vee^\perp \\
				0&  (-1)^{p+1}P_-^\perp b T P_+^\perp &0    & \lambda P_-^\perp TP_\vee^\perp\\
				P_\vee^\perp bT& (-1)^{p+1}\mu  P_\vee^\perp TP_+^\perp                                & (-1)^{p+1}\lambda P_\vee^\perp TP_-^\perp                      & 0
			\end{pmatrix}
\begin{pmatrix}
v_1\\
v_2\\
v_3\\
v_4
\end{pmatrix},
\end{eqnarray*}
and
\begin{eqnarray*}
&&\hT_+(0)\hb_+(\lambda,\mu)V_p \\
&=& \begin{pmatrix}
				0&0 & 0&   T P_\vee^\perp   \\
				0& 0&  (-1)^{p-1}P_+^\perp T  P_-^\perp  & 0\\
				0&  (-1)^{p}P_-^\perp T P_+^\perp &0 & 0\\
				P_\vee^\perp T& 0& 0& 0
			\end{pmatrix}\hb_+(\lambda,\mu)V_p\\
&=& \begin{pmatrix}
				0&0                   &  0                                & Tb^* P_\vee^\perp   \\
				0& 0                  &  (-1)^{p-1}P_+^\perp Tb^*  P_-^\perp  & (-1)^{p-1}\mu P_+^\perp TP_\vee^\perp \\
				0&  (-1)^{p}P_-^\perp Tb^* P_+^\perp &0    & (-1)^{p-1}\lambda P_-^\perp TP_\vee^\perp\\
				P_\vee^\perp Tb^*& \mu  P_\vee^\perp TP_+^\perp                                & \lambda P_\vee^\perp TP_-^\perp                      & 0
			\end{pmatrix}
\begin{pmatrix}
v_1\\
v_2\\
v_3\\
v_4
\end{pmatrix}.
\end{eqnarray*}
Now the equality
$$(-1)^p\hb_+(\lambda,\mu)\hT_+(0) +\hT_+(0)\hb^*_+(\lambda,\mu)=0$$
follows.

		 At last, we show that $\hT_+(0) $ is a homotopy equivalence. In fact, we decompose $\hH_{+,*,X}$ as
$\mathcal{H}_{1,*}\oplus \mathcal{H}_{2,*}$, where
		\[\mathcal{H}_{1,*}=  H_{*,X}\oplus P_+^\perp H_{*+1,X},\ \  \mathcal{H}_{2,*} = P_-^\perp H_{*+1,X}\oplus P_\vee^\perp H_{*+2,X}. \]
		Set
		\[b_1=\left(
		\begin{array}{cc}
		b_{\widetilde{X}} &0  \\
		\mu P_+^\perp   &   -P_+^\perp  b_{\widetilde{X}}  \\
		\end{array}
		\right), \ \
		b_2=
		\left(
		\begin{array}{cc}
		-P_-^\perp b_{\widetilde{X}} &0 \\
		-\mu P_\vee^\perp &  P_\vee^\perp b_{\widetilde{X}}
		\end{array}
		\right).
		\]
		Set
		\[
		T_1= \left(
		\begin{array}{cc}
		0&  (-1)^{p}P_-^\perp T P_+^\perp \\
		P_\vee^\perp T & 0
		\end{array}
		\right),\ \
        T_2 = \left(
		\begin{array}{cc}
		0&  T  P_\vee^\perp \\
		 (-1)^{p-1} P_+^\perp T  P_-^\perp  & 0\\
		\end{array}
		\right).
		\]
		It is direct to see that we have
		\[
		\xymatrix{
			0 \ar[r] & (\mathcal{H}_{1,*}, b_1)^G \ar[r] \ar[d]^{T_1} &  (\hH_{+,*,X}, b_{\lambda, \mu})^G\ar[r] \ar[d]^{\hT_+(0) } & (\mathcal{H}_{2,*}, b_2)^G \ar[r]\ar[d]^{T_2} & 0\\
			0 \ar[r] & (\mathcal{H}^*_{1,*}, b^*_2)^G \ar[r] & (\hH^*_{+,*,X}, b^*_{\lambda, \mu})^G \ar[r] &  (\mathcal{H}^*_{1,*}, b^*_1)^G \ar[r] & 0.
		}
		\]
		By basic topology theory, we know that $T_1:(\mathcal{H}_{1,*}, b_1) \to  (\mathcal{H}^*_{2,*}, b^*_2) $ and $T_2:(\mathcal{H}_{2,*}, b_2)\to  (\mathcal{H}^*_{1,*}, b^*_1) $ are both $G$-equivariant analytically controlled chain homotopy equivalences, so be
		$\hT_+(0) $ by Lemma 4.2 of \cite{HR051}.

The $s\neq 0$ case follows from almost the same but much simpler computation. The proof is then completed.
	\end{proof}
	%\begin{remark}
	%Intuitively, $T'_{s, +,\widetilde{X}}$ should has the form
	%\[
	% \left(
	%    	\begin{array}{cccc}
	%    		0&0 & 0&   P_\vee^\perp T   \\
	%    		0& 0&  (-1)^{p}P_-^\perp T  P_+^\perp  & 0\\
	%    		0&  (-1)^{p} P_+^\perp T P_-^\perp &0 & 0\\
	%    		(-1)^{p} (-1)^{p+1} TP_\vee^\perp & 0& 0& 0
	%    	\end{array}
	%    	\right),
	%\]
	%however, the orientation of $\partial_\pm M$ inherit from $M$ reverse to each other.
	%\end{remark}
	In the same reason, we have

\begin{lemma}\label{lem 4 complex negative}
Let
\[
(H_{*,X}, b, T, P_\pm)^G
\]
be a $n+2$-dimensional $G$-equivariant analytically controlled Hilbert-Poincar\'e triple over $X$.
Set
		\begin{eqnarray*}
			 \widehat{H}_{-,*,X} &=& H_{*, X}\oplus P_-^\perp H_{*+1,X} \oplus P_+^\perp H_{*+1, X}\oplus P_\vee^\perp H_{*+2,X},  *=0,1,\cdots n,\\
			\widehat{b}_-(\lambda,\mu)&=&\left(
			\begin{array}{cccc}
				b &0 & 0&0 \\
				\mu P_-^\perp   &   -P_-^\perp  b& 0&0 \\
				\lambda P_+^\perp  & 0 & -P_+^\perp b &0 \\
				0	& \lambda P_\vee^\perp & -\mu P_\vee^\perp &  P_\vee^\perp b
			\end{array}
			\right)
\end{eqnarray*}
and \[ \widehat{T}_{-}(s)  =   \left(
			\begin{array}{cccc}
				0&0 & 0& e^{-i\pi s} P_\vee^\perp T   \\
				0& 0&  (-1)^{p-1}e^{-i\pi s}P_-^\perp T  P_+^\perp  & 0\\
				0&  (-1)^{p}e^{i\pi s}P_+^\perp T P_-^\perp &0 & 0\\
				e^{i\pi s}TP_\vee^\perp & 0& 0& 0
			\end{array}
			\right)
		\]
on $ \widehat{H}_{-,p,X}.$ Then $(\widehat{H}_{-,*,X}, \widehat{b}_-(\lambda, \mu),  \widehat{T}_{-}(s))^G$
 defines an $n$-dimensional $G$-equivariant analytically controlled Hilbert-Poincar\'e complex over $X$ as long as
		\begin{enumerate}
			\item $ \lambda,\mu\in[-1,0], s \in [0,1]$.
			\item $\lambda s=0.$
%\item $\lambda=0$ if $\mu =0$.
		\end{enumerate}
	\end{lemma}	
	\begin{proof}
		It is sufficient to prove that
\[
(\widehat{H}_{-,*,X}, \widehat{b}_-(\lambda, \mu),  -\widehat{T}_{-}(s))^G
\]
defines an $n$-dimensional $G$-equivariant analytically controlled Hilbert-Poincar\'e complex over $X$. However,
The lemma follows from Lemma \ref{lem 4 complex positive} and a unitary equivalence between
\[(\widehat{H}_{+,*,X}, \widehat{b}_+(\lambda, \mu),  \widehat{T}_{+}(s))^G\]
and
\[
(\widehat{H}_{-,*,X}, \widehat{b}_-(\mu, \lambda),  -\widehat{T}_{-}(s))^G
\]
induced by
		\[
		\left(
		\begin{array}{cccc}
		1 & 0 & 0 & 0 \\
		0 & 0 & 1 & 0 \\
		0 & 1 & 0 & 0 \\
		0 & 0 & 0 & -1
		\end{array}
		\right).
		\]
	\end{proof}
	
% We mention that  when $(H_{*,X}, b, T,P_\pm)^G$ is $2k$ dimensional,  then the elements will be denoted by
%	\[
%	P_+(B_{\mu,P_+}+ S_{\mu,P_+})- P_+({\mu,B_{P_+}-S_{\mu,P_+}})\in C_0( [1,0), C^*(\widetilde{X})^G)
%	\]
%	and
%	\[
%	P_+(B_{\mu,P_-}+ S_{\mu,P_-})- P_+({B_{\mu,P_-}-S_{\mu,P_-}})\in C_0( [1,0), C^*(\widetilde{X})^G)
%	\]
%	similarly.
	
	\begin{lemma}\label{lem rel homotopy equivalence positive}
		Let
\[
(H_{*,X}, b, T, P_\pm)^G
\]
be an $n+2$-dimensional $G$-equivariant analytically controlled Hilbert-Poincar\'e triple over $X$.
Set		
\begin{eqnarray*}
			\overline{P_+H}_{*,X} &=& P_+H_{*,X} \oplus P_\wedge^\perp P_+ H_{*+1,X}, *=0,\cdots n,\\
			\overline{P_+b}(\lambda)&=& \left(
			\begin{array}{cc}
				P_+ b & 0\\
				\lambda P_\wedge^\perp  & -P_\wedge^\perp P_+ b
			\end{array}
			\right) ,\lambda\in [-1,0] \\
		\end{eqnarray*}
and \[
	\overline{P_+ T_0P_+}(s) =
			\left(
			\begin{array}{cc}
				0 &  e^{is\pi}(P_+ T_0P_+)P_\wedge^\perp    \\
				(-1)^pe^{-is\pi} P_\wedge^\perp (P_+ T_0P_+)   & 0
			\end{array}
			\right), s\in [0,1].
\]
 Then the $n$-dimensional $G$-equivariant analytically controlled Hilbert-Poincar\'e complex over $X$.
\[
 (\overline{P_+H}_{*,X},\overline{P_+b}(\lambda), \overline{P_+ T_0P_+}(s))^G ,\lambda s=0,
\]
 is $G$-equivariantly homotopy equivalent to the complex
\[(\widehat{H}_{+,*,X}, \widehat{b}_+(\lambda, -1),  \widehat{T}_{+}(s))^G,\lambda s=0\]
defined in Lemma \ref{lem 4 complex positive},
under the the controlled chain map
		\begin{eqnarray*}
			A:\overline{P_+H}_{*,X} &\to &      	\widehat{H}_{+,*,X}\\
			(v, w) &\to & (v, 0, w, 0)
		\end{eqnarray*}
	\end{lemma}
	\begin{proof}
		By basic facts about mapping cone complex, one can see that
\[
A:  (\overline{P_+H}_{*,X},\overline{P_+b}(\lambda))^G\to (\widehat{H}_{+,*,X}, \widehat{b}_+(\lambda, -1))^G
\]
 is a $G$-equivariant analytically controlled  homotopy equivalence. It remains to show that $A \overline{P_+ T_0P_+}(s)A^* $ and  $ \widehat{T}_{+}(s)$ are $G$-equivariant analytically controlled homotopy equivalent to each other. However, this can be seen by simply verifying
		\[
		A\overline{P_+ T_0P_+}(s) A^*-\widehat{T}_{+}(s) = h_{p+1}   \widehat{b}_+(\lambda, -1) + (-1)^p\widehat{b}_+(\lambda, -1) h_p,
		\]
		where the operator $h_p $ is an analytically controlled operator on $\widehat{H}_{+,p,X},$
		which is defined as
		\[
		\left(
		\begin{array}{cccc}
		0 & 0& P_+ T P_-^\perp & 0 \\
		0 & 0& 0 & 0 \\
		(-1)^pP_-^\perp T P_+ & 0& 0 & 0 \\
		0 & 0& 0 & 0
		\end{array}
		\right).
		\]
%		Note that by Remark \ref{remark for dual of b and projection}, we have
%		\begin{eqnarray*}
%			&&bP_+TP^\perp_- + (-1)^p P_+TP^\perp_- b^*\\
%			&=& P_+bTP^\perp_- + (-1)^p P_+Tb^*P^\perp_- \\
%			&=& P_+T_0 P_+P^\perp_-.
%		\end{eqnarray*}
	\end{proof}
	
	\begin{corollary}\label{cor homotopy factor through positive}
		The boundary complex
\[(P_\wedge (P_+H_{*,X}), P_\wedge(P_+ b), (P_+ T_0P_+)_0)^G\]
of the $G$-equivariant analytically controlled Hilbert-Poincar\'e pair
\[
(P_+H_{*,X},P_+ b, P_+ T_0P_+, P_\wedge )^G,
\]
 is $G$-equivariantly homotopy equivalent to the complex
\[(\widehat{H}_{+,*,X}, \widehat{b}_+(-1, -1),  \widehat{T}_{+}(s))^G,\]
 with the homotopy factors through the $G$-equivariant homotopy equivalence between
\[(P_\wedge (P_+H_{*,X}), P_\wedge(P_+ b), (P_+ T_0P_+)_0)^G\]
and
\[(\overline{P_+H}_{*,X},\overline{P_+b}(-1),  \overline{P_+ T_0P_+}(s))^G.\]
	\end{corollary}
	In the same reason, we have
	\begin{lemma}\label{lem rel homotopy equivalence negative}
		Let
\[
(H_{*,X}, b, T, P_\pm)^G
\]
be an $n+2$-dimensional $G$-equivariant analytically controlled Hilbert-Poincar\'e triple over $X$.
Set		
\begin{eqnarray*}
			\overline{P_-H}_{*,X} &=& P_-H_{*,X} \oplus P_\wedge^\perp P_- H_{*+1,X}, *=0,\cdots n.\\
			\overline{P_-b}(\lambda)&=& \left(
			\begin{array}{cc}
				P_- b & 0\\
				\lambda P_\wedge^\perp  & -P_\wedge^\perp P_- b
			\end{array}
			\right) ,\lambda\in [-1,0] \\
		\end{eqnarray*}
and \[
	\overline{P_- T_0P_-}(s) =
			\left(
			\begin{array}{cc}
				0 &  e^{is\pi}(P_- T_0P_-)P_\wedge^\perp    \\
				(-1)^pe^{-is\pi} P_\wedge^\perp (P_- T_0P_-)   & 0
			\end{array}
			\right), s\in [0,1].
\]
 Then the $n$-dimensional $G$-equivariant analytically controlled Hilbert-Poincar\'e complex over $X$.
\[
 (\overline{P_-H}_{*,X},\overline{P_-b}(\lambda), \overline{P_- T_0P_-}(s))^G ,\lambda s=0,
\]
 is $G$-equivariantly homotopy equivalent to the complex
\[(\widehat{H}_{-,*,X}, \widehat{b}_-(\lambda, -1),  \widehat{T}_{-}(s))^G,\lambda s=0,\]
defined in Lemma \ref{lem 4 complex positive},
under the the controlled chain map
		\begin{eqnarray*}
			A:\overline{P_-H}_{*,X} &\to &      	\widehat{H}_{-,*,X}\\
			(v, w) &\to & (v, 0, w, 0)
		\end{eqnarray*}
	\end{lemma}
	
	\begin{corollary}\label{cor homotopy factor through negative}
		The boundary complex
\[(P_\wedge (P_-H_{*,X}), P_\wedge(P_- b), (P_- T_0P_-)_0)^G\]
of the $G$-equivariant analytically controlled Hilbert-Poincar\'e pair
\[
(P_-H_{*,X},P_- b, P_- T_0P_-, P_\wedge )^G,
\]
 is $G$-equivariantly homotopy equivalent to the complex
\[(\widehat{H}_{-,*,X}, \widehat{b}_-(-1, -1),  \widehat{T}_{-}(s))^G,\]
 with the homotopy factors through the $G$-equivariant homotopy equivalence between
\[(P_\wedge (P_-H_{*,X}), P_\wedge(P_- b), (P_- T_0P_-)_0)^G\]
and
\[(\overline{P_-H}_{*,X},\overline{P_-b}(-1),  \overline{P_- T_0P_-}(s))^G.\]
	\end{corollary}

\subsection{Bordism invariance of the relative signature of compact  PL manifolds with boundary}

In this subsection, we formulate the bordism invariance of the relative signature of compact PL manifolds with boundary, whose  proof is almost immediate due to the preparation in Subsection \ref{subsec hil poin triple}.

Let $(M, \partial_\pm M)$ be an $n+1$-dimensional compact PL manifold 2-ads, with $\pi_1(M)=\Gamma$,  $\pi_1(\partial_\pm M)= \Gamma_\pm$ and $\pi_1(\partial \partial_\pm M)=G$. Let $\widetilde M,$  $\widetilde{\partial_{\pm}M}$ and $\widetilde{\partial\partial_{\pm}M}$ be the universal covering space of $M$, $\partial M,$ $\partial_{\pm} M$ and $\partial \partial_{\pm} M$ be the universal covering space of $M,\  \partial_{\pm} M,\ \partial \partial_{\pm} M $ respectively. Then as in Subsection \ref{subsec relative C alg}, we have relative $C^*$-algebras $C^*(\widetilde{\partial_+M}, \widetilde{\partial \partial_{+} M})^{\Gamma_+, G},$ $C^*(\widetilde{M}, \widetilde{\partial_{-} M})^{\Gamma, \Gamma_-},$
and a  $C^*$-homomorphism
\[
\iota_+: C^*(\widetilde{\partial_+M}, \widetilde{\partial \partial_{+} M})^{\Gamma_+, G}\to C^*(\widetilde{M}, \widetilde{\partial_{-} M})^{\Gamma, \Gamma_-}.
\]

\begin{theorem}\label{theo bordism invariance of relative signature}
		Let $(M, \partial_\pm M)$ be an $n+1$-dimensional compact PL manifold 2-ads, with $\pi_1(M)=\Gamma$, $\pi_1(\partial_\pm M)= \Gamma_\pm$ and $\pi_1(\partial \partial_\pm M)=G$. Let $\iota_+: (\partial_+ M, \partial \partial_+ M )\to (M, \partial_- M)$ be the embedding of the positive part of the boundary. Then we have
		\[
		\iota_+^* ({\rm relInd}(\partial_+M, \partial \partial_+ M ) )=0\in K_{n}(C^*(\widetilde{M}, \widetilde{\partial_- M})^{\Gamma, \Gamma_-} ).
		\]
	\end{theorem}

%Let $(H_{*,X}, b, T,P_\pm)^\Gamma$ be a $2k+1$-dimensional $G$-equivariant analytically controlled Hilbert-Poincar\'e triple. Then we can define an element in $[1,0)\times C^*(\widetilde{X})^G$  by considering the signature es of the  family of $2k-1$ dimensional $G$-equivariant analytically controlled Hilbert-Poincar\'e complexes
%	\[
%	\left\{
%	\begin{array}{cc}
%	(\hH_{+,*,X}, \hb_+(-(1-2t), \mu),  \hT_{+}(0))^G & t\in[0, \frac{1}{2}]\\
%	(\hH_{+,*,X}, \hb_+(0, \mu) , \hT_{+}(2t-1))^G & t\in[\frac{1}{2}, 1]
%	\end{array}
%	\right.
%	\]
%	and
%	\[
%\left\{
%	\begin{array}{cc}
%	(\hH_{-,*,X}, \hb_-(-(1-2t), \mu),  \hT_{-}(0))^G & t\in[0, \frac{1}{2}]\\
%	(\hH_{-,*,X}, \hb_-(0, \mu) , \hT_{-}(2t-1))^G  & t\in[\frac{1}{2}, 1]
%	\end{array}
%	\right.
%	\]
%     which  will be  denoted by
%	\[
%	\frac{B_{\mu,P_+}+ S_{\mu,P_+}}{B_{\mu,P_+}-S_{\mu,P_+}}\in C_0( [1,0), C^*(\widetilde{X})^G)
%	\]
%	and
%	\[
%	\frac{B_{\mu,P_-}+ S_{\mu,P_-}}{B_{\mu,P_-}-S_{\mu,P_-}}\in C_0( [1,0), C^*(\widetilde{X})^G)
%	\]
%	respectively.
	
	\begin{proof}
Set
\[
\iota_-: C^*(\widetilde{\partial_-M}, \widetilde{\partial \partial_{-} M})^{\Gamma_-, G}\to C^*(\widetilde{M}, \widetilde{\partial_{-} M})^{\Gamma, \Gamma_-}.
\]
Note that by the definition of the relative $C^*$-algebras, we have
\[
\iota^*_-(K_{n}(C^*(\widetilde{\partial_-M}, \widetilde{\partial \partial_{-} M})^{\Gamma_-, G}))=\{0\},
\]
thus
\[
\iota^*_-({\rm relInd}(\partial_-M, \partial \partial_- M ) )=0\in K_n(C^*(\widetilde{M}, \widetilde{\partial_{-} M})^{\Gamma, \Gamma_-}).
\]
Hence it is sufficient to show that
\[
\iota_+^* ({\rm relInd}(\partial_+M, \partial \partial_+ M ) )=-\iota^*_-({\rm relInd}(\partial_-M, \partial \partial_- M ) ).
\]
%In the following,  We will go through the detail for $n=2k+1$ case only. The $n=2k$ case is totally the same.

Let $\text{Tri}(M, \partial_{\pm} M)$ be a triangulation of $(M, \partial_{\pm} M)$, then it induces a $\Gamma$-equivariant analytically controlled Hilbert complex over $\widetilde M,$ denoted as $(H_{*,\tM}, b)^\Gamma.$  Let $T$ be its Poincar\'e duality, and $P_{\pm}$ be the usual projections on to the subspace of $H_{*,X}$ spanned by complex on $\partial_{\pm} M$ respectively. Then
\[
(H_{*,\tM}, b_{\widetilde M}, T, P_{\pm})^\Gamma
\]
is a $\Gamma$-equivariant analytically controlled Hilbert-Poincar\'e triple over $\widetilde M.$ Parallelly, we have
\[
(H_{*, \widetilde{\partial_\pm M}}, b_{\widetilde{\partial_\pm M}}, T_{\partial_\pm}, P_{\pm})^{\Gamma_{\pm}},
\]
the $\Gamma_\pm$-equivariant analytically controlled Hilbert-Poincar\'e pairs over $\widetilde{\partial_{\pm} M},$ and
\[
(H_{*,{\widetilde{\partial \partial_{\pm}M }}}, b_{\widetilde{\partial\partial_{\pm} M}}, T_{\partial\partial_{\pm}})^{G},
\]
the $G$-equivariant analytically controlled Hilbert-Poincar\'e complexes over $\widetilde{\partial \partial_{\pm} M}.$

Set
\[[(a_+,f_+)]=    {\rm relInd}(\partial_+M, \partial \partial_+ M ),\ [(a_-,f_-)]=    {\rm relInd}(\partial_-M, \partial \partial_- M ).\]
Then $a_\pm$ are the representatives of the signatures of
\[
(H_{*,{\widetilde{\partial \partial_{\pm} M}}}, b_{\widetilde{\partial\partial_{\pm} M}}, T_{\partial\partial_{\pm}})^{G}
\]
defined in Theorem \ref{def signature of hp complex}. Note that $\iota_-(a_-)$
equals the representative of the signature of the complex defined in line \ref{eq rel homology complex}
\[
(\bH_{*,{\widetilde{\partial_{-} M}}}, \bb_{\widetilde{\partial_{-}M}}, T_{\partial\partial_{-},0})^{\Gamma_-},
\]
and due to the orientation being opposite, $
\iota_+(a_+)$
equals the element representing the signature of the complex defined in line \ref{eq rel homology complex}
\[
(\bH_{*,{\widetilde{\partial_{-} M}}}, \bb_{\widetilde{\partial_{-}M}}, -T_{\partial\partial_{-},0})^{\Gamma_-}.
\]

On the other hand,
\[
\iota_\pm (f_\pm)
\]
equals the path, as defined in line \ref{eq odd sig bord path}, derived from the complex
\[(\widehat{H}_{\pm,*,\widetilde M}, \widehat{b}_{\widetilde M,\pm}(\lambda,-1), \widehat{T}_{\pm}(s))^\Gamma, \lambda\in [-1,0], s\in [0,1], \lambda s=0.\]
Now the fact that
\[
\iota_+ (f_+)=-\iota_- (f_-)
\]
follows from Lemma \ref{lem rel homotopy equivalence positive}, Corollary \ref{cor homotopy factor through positive}, Lemma \ref{lem rel homotopy equivalence negative}, Corollary \ref{cor homotopy factor through negative}, and the fact that the complex
\[
(\widehat{H}_{+,*,\widetilde M}, \widehat{b}_{\widetilde M,+}(\lambda, \lambda),  \widehat{T}_{+}(s))^\Gamma, \lambda s=0
\]
differs from the complex
\[
(\widehat{H}_{-,*,\widetilde M}, \widehat{b}_{\widetilde M,-}(\lambda, \lambda),  -\widehat{T}_{-}(s))^\Gamma
, \lambda s=0\]
by the conjugation of the unitary
\[
		U= \left(
		\begin{array}{cccc}
		1 & 0 & 0 & 0\\
		0 & 0 & 1 & 0\\
		0 & 1 & 0 & 0\\
		0 & 0 & 0 & -1
		\end{array}
		\right),
		\]
		This finishes our proof.
	\end{proof}

\subsection{Group homomorphism from the relative $L$-theory to the $K$-theory}

In this section, we show that the relative signature of compact PL manifolds with boundary induces an additive map from the relative $L$-theory to the $K$-theory.

Let  $(X, \partial X)$ be an $n$-dimensional compact PL manifold with boundary. Set $\Gamma=\pi_1(X),$ $G=\pi_1(\partial X).$ Let
\[
\theta =(M, \partial_{\pm }M,\phi,N, \partial_{\pm} N, \psi, f)
\]
be an element in $L_n(\pi_1(X),\pi_1(\partial X)).$ Then let
\[
(M\cup_f -N,\partial_-M\cup_f -\partial_- N)
\]
be the space obtained by glueing $(M, \partial_{\pm} M)$ and $(-N, -\partial_{\pm}N)$ by the homotopy equivalence $f$. Although
\[
(M\cup_f -N,\partial_-M\cup_f -\partial_- N)
\]
is not a compact $PL$ manifold with boundary in general, one can still consider the  Poincar\'e duality operator induced by the cap product with the fundamental class $[M\cup_f -N],$ and projections onto
\[
\partial_-M\cup_f -\partial_- N.
\]
Thus it  makes sense to consider the relative signature
\[
\text{relInd}(M\cup_f -N,\partial_-M\cup_f -\partial_- N)\in K_n(C^*(\widetilde X, \widetilde{\partial X})^{\Gamma,G}).
\]

\begin{definition}
For each element
\[\theta =(M, \partial_{\pm }M,\phi,N, \partial_{\pm} N, \psi, f),\]
define
\[
\text{relInd}(\theta)=\text{relInd}(M\cup_f -N,\partial_-M\cup_f -\partial_- N)\in K_n(C^*(\widetilde X, \widetilde{\partial X})^{\Gamma,G}).
\]
\end{definition}

By the bordism invariance of the relative signature, and the fact that it is additive on disjoint unions, we have the following theorem.
\begin{theorem}\label{add map from L to K}
The map
\[
\text{relInd}: L_n(\pi_1(X),\pi_1(\partial X))\to K_n(C^*(\widetilde X, \widetilde{\partial X})^{\Gamma,G})
\]
is a well defined group homomorphism.
\end{theorem}

%%%%%%%%%%%%%%%%%%%%%%%%%%%

\newpage
%\setlength{\marginrulewidth}{0pt}

%\bibliography{ijmsample}
%\bibliographystyle{ijmart}

\end{document}